\documentclass[11pt,letterpaper]{amsart}
\usepackage{amsfonts,latexsym,amsthm,amssymb,graphicx}
\usepackage[all]{xy}
\usepackage[usenames]{color} 
\usepackage{xypic,amscd,epsfig}
\usepackage{hyperref}
\usepackage{pdfsync}
\usepackage[capitalize,nameinlink]{cleveref}

\usepackage{mathtools}
\makeatletter
\DeclareRobustCommand\widecheck[1]{{\mathpalette\@widecheck{#1}}}
\def\@widecheck#1#2{%
    \setbox\z@\hbox{\m@th$#1#2$}%
    \setbox\tw@\hbox{\m@th$#1%
       \widehat{%
          \vrule\@width\z@\@height\ht\z@
          \vrule\@height\z@\@width\wd\z@}$}%
    \dp\tw@-\ht\z@
    \@tempdima\ht\z@ \advance\@tempdima2\ht\tw@ \divide\@tempdima\thr@@
    \setbox\tw@\hbox{%
       \raise\@tempdima\hbox{\scalebox{1}[-1]{\lower\@tempdima\box
\tw@}}}%
    {\ooalign{\box\tw@ \cr \box\z@}}}
\makeatother

\crefname{conjecture}{Conjecture}{Conjectures}
\Crefname{conjecture}{Conjecture}{Conjectures}

\DeclareFontFamily{OT1}{rsfs}{}
\DeclareFontShape{OT1}{rsfs}{n}{it}{<-> rsfs10}{}
\DeclareMathAlphabet{\mathscr}{OT1}{rsfs}{n}{it}

\CompileMatrices

\newtheorem{theorem}{Theorem}[section]
\newtheorem{lemma}[theorem]{Lemma}
\newtheorem{corol}[theorem]{Corollary}
\newtheorem{prop}[theorem]{Proposition}

{\theoremstyle{definition} \newtheorem{defin}[theorem]{Definition}}
{\theoremstyle{remark} \newtheorem{remark}[theorem]{Remark}
\newtheorem{example}[theorem]{Example}}

\newcommand{\Abb}{{\mathbb A}}
\newcommand{\C}{{\mathbb C}}
\newcommand{\Nbb}{{\mathbb{N}}}
\newcommand{\Pbb}{{\mathbb{P}}}
\newcommand{\T}{{\mathbb T}}
\newcommand{\Z}{{\mathbb Z}}
\newcommand{\caD}{{\mathcal D}}
\newcommand{\cF}{{\mathcal F}}
\newcommand{\caH}{{\mathcal H}}
\newcommand{\caL}{{\mathcal L}}
\newcommand{\cM}{{\mathcal M}}
\newcommand{\cS}{{\mathcal S}}
\newcommand{\cT}{{\mathcal T}}
\newcommand{\g}{\mathfrak{g}}
\newcommand{\mft}{\mathfrak{t}}
\newcommand{\mfs}{\mathfrak{s}}
\newcommand{\cO}{\mathscr O}
\newcommand{\aP}{\alpha}
\newcommand{\aX}{{\underline\alpha}}
\newcommand{\aPp}{{C(\alpha)}}
\newcommand{\csm}{{c_{\mathrm{SM}}}}
\newcommand{\csmve}{{c_{\mathrm{SM}}^\vee}}
\newcommand{\csT}{c^T_*}
\newcommand{\chcT}{\chc_*^T}
\newcommand{\csmT}{{c^T_{\mathrm{SM}}}}
\newcommand{\csmTv}{{c^{T,\vee}_{\mathrm{SM}}}}
\newcommand{\csmTh}{{c^{T,\hbar}_{\mathrm{SM}}}}
\newcommand{\csmThv}{{c^{T,\hbar,\vee}_{\mathrm{SM}}}}
\newcommand{\ssm}{{s_{\mathrm{SM}}}}
\newcommand{\ssmT}{{s_{\mathrm{SM}}^T}}
\newcommand{\chs}{{\widecheck{s}}}
\newcommand{\chssm}{{\chs_{\mathrm{SM}}}}
\newcommand{\chssmT}{{\chs_{\mathrm{SM}}^T}}
\newcommand{\cMa}{{c_{\mathrm{Ma}}}}
\newcommand{\cMaT}{{c_{\mathrm{Ma}}^{T}}}
\newcommand{\cMaTv}{{c^{T,\vee}_{\mathrm{Ma}}}}
\newcommand{\one}{1\hskip-3.5pt1}
\newcommand{\chc}{{\widecheck{c}}}
\newcommand{\che}{{\widecheck{e}}}
\newcommand{\qede}{\hfill$\lrcorner$}
\newcommand{\Til}{\widetilde}
\newcommand{\cc}{c}
\newcommand{\cFTinv}{\cF^T_\mathrm{inv}}

\DeclareMathOperator{\id}{id}
\DeclareMathOperator{\rk}{rk}
\DeclareMathOperator{\CC}{CC}
\DeclareMathOperator{\DR}{DR}
\DeclareMathOperator{\Char}{Char}
\DeclareMathOperator{\Eu}{Eu}
\DeclareMathOperator{\cEu}{\widecheck{\mathrm{E}}u}
\DeclareMathOperator{\Shadow}{Shadow}
\DeclareMathOperator{\cn}{cn}
\DeclareMathOperator{\stab}{stab}
\DeclareMathOperator{\codim}{codim}
\DeclareMathOperator{\Fl}{Fl}
\DeclareMathOperator{\Frac}{Frac}
\DeclareMathOperator{\Segre}{s}
\DeclareMathOperator{\supp}{supp}
\DeclareMathOperator{\Perv}{Perv}
\DeclareMathOperator{\Mod}{Mod}
\DeclareMathOperator{\KL}{KL}
\DeclareMathOperator{\IC}{IC}
\DeclareMathOperator{\Sym}{Sym}
\DeclareMathOperator{\pt}{pt}

\begin{document}
\title[Characteristic cycles, CSM classes, and positivity]{Shadows of characteristic cycles, Verma modules, and positivity of Chern-Schwartz-MacPherson classes of Schubert cells}

\author[P. Aluffi]{Paolo Aluffi}
\address{
Mathematics Department, 
Florida State University,
Tallahassee FL 32306
USA
}
\email{aluffi@math.fsu.edu}

\author[L. Mihalcea]{Leonardo C.~Mihalcea}
\address{
Department of Mathematics, 
Virginia Tech University, 
Blacksburg, VA 24061
USA
}
\email{lmihalce@vt.edu}

\author[J. Sch\"urmann]{J\"org Sch\"urmann}
\address{Mathematisches Institut, Universit\"at M\"unster, Germany}
\email{jschuerm@uni-muenster.de}

\author[C. Su]{Changjian Su}
\address{Yau Mathematical Sciences Center, Tsinghua University, Beijing, China}
\email{changjiansu@mail.tsinghua.edu.cn}


\thanks{P. Aluffi was supported in part by NSA Award H98230-16-1-0016 and Simons 
Collaboration Grants \#245301 and \#625561; L.~C.~Mihalcea was supported in part by 
NSA Young Investigator Award H98320-16-1-0013, the Simons Collaboration Grant \#581675, and 
the NSF award DMS-2152294;
J. Sch\"{u}rmann was funded by the Deutsche Forschungsgemeinschaft (DFG, German Research Foundation) Project-ID 427320536 -- SFB~1442, as well as under Germany's Excellence Strategy EXC 2044 390685587, Mathematics M\"unster: Dynamics -- Geometry -- Structure.}

\begin{abstract} 
  Chern-Schwartz-MacPherson (CSM) classes generalize to singular
  varieties the classical total homology Chern class of the tangent
  bundle of a smooth compact complex algebraic variety. The theory of CSM
  classes has been extended to the equivariant setting by Ohmoto. We
  prove that for an arbitrary complex algebraic variety $X$, the
  homogenized, torus equivariant CSM class of a constructible function
  $\varphi$ is the restriction of the characteristic cycle of
  $\varphi$ via the zero section of the cotangent bundle of $X$. In
  the process we relate the CSM class in question to a Segre operator
  applied to the characteristic cycle. This extends to the equivariant
  setting results of Ginzburg and of Sabbah. We specialize $X$ to be a
  (generalized) flag manifold $G/B$. In this case CSM classes are
  determined by a Demazure-Lusztig (DL) operator. We prove a `Hecke
  orthogonality' of CSM classes, determined by the DL operator and its
  adjoint, and a `geometric orthogonality' between CSM and
  Segre-MacPherson classes. This implies a remarkable formula for the
  CSM class of a Schubert cell in terms of the Segre class of the
  characteristic cycle of a holonomic Verma $\caD_X$-module. We
  deduce a positivity property for CSM classes previously conjectured
  by Aluffi and Mihalcea, and extending positivity results by J.~Huh
  in the Grassmann manifold case. As an application, we prove
  positivity for certain Kazhdan-Lusztig classes, and for some
  instances of Mather classes, of Schubert varieties. We also
  establish an equivalence between CSM classes and stable envelopes;
  this reproves results of Rim{\'a}nyi and Varchenko.  Finally, we
  generalize all of this to partial flag manifolds~$G/P$.
\end{abstract}
\maketitle

\tableofcontents


\section{Introduction}  
According to a conjecture attributed to Deligne and Grothendieck,
there is a unique natural transformation $c_*: \cF \to H_*$ from the
functor of constructible functions to homology, over the category
of compact complex algebraic varieties and proper morphisms,
such that if $X$ is
smooth then $c_*(\one_X)=c(TX)\cap [X]$.  This conjecture was proved
by MacPherson~\cite{macpherson:chern}; the class $c_*(\one_X)$ for
possibly singular $X$ was later shown to coincide with a class defined
earlier by M.-H.~Schwartz \cite{schwartz:1, schwartz:2, BS81}. For any
constructible subset $W\subseteq X$, we call the class
$c_{SM}(W):= c_*(\one_W)\in H_*(X)$ the
\textit{Chern-Schwartz-MacPherson} (CSM) class of $W$ in~$X$. The
theory of CSM classes was later extended to the equivariant setting by
Ohmoto~\cite{ohmoto:eqcsm}. We denote by $\csmT(W):= \csT(\one_W)$
the torus equivariant CSM class.

The main objects of study in this paper are the (torus) equivariant CSM
classes of Schubert cells in flag manifolds. These classes were
computed in various generality: for Grassmannians, in the
non-equivariant specialization, by Aluffi and Mihalcea
\cite{aluffi.mihalcea:csm,mihalcea:binomial}, and Jones
\cite{jones:csm}; for type A partial flag manifolds by Rim{\'a}nyi and
Varchenko \cite{rimanyi.varchenko:csm}, using the fact that they
coincide with certain weight functions studied in
\cite{RTV:partial,RTV:Kstable,RTV:Yangian}, and using interpolation
properties obtained by Weber \cite{Weber}; and for flag manifolds in
all Lie types by Aluffi and Mihalcea~\cite{aluffi.mihalcea:eqcsm}
using Bott-Samelson desingularizations of Schubert varieties.

One of our main goals is to show that the CSM classes of Schubert cells are
effective, thus proving the non-equivariant version of a positivity
conjecture from \cite{aluffi.mihalcea:eqcsm}.  This generalizes a
similar positivity result proved by J.~Huh \cite{huh:positivity} for
the Grassmannian case (in turn proving an earlier conjecture posed in
\cite{aluffi.mihalcea:csm}).  This also implies positivity of the
Kazhdan-Lusztig classes (associated to the stalk Euler characteristic
of intersection cohomology complexes of Schubert varieties), and in
some instances positivity of Mather classes.

On the route of proving the positivity conjecture we revisit, and
extend to the equivariant setting, the `Lagrangian model' for
constructing MacPherson's transformation for arbitrary smooth projective
varieties $X$, as developed by Sabbah \cite{sabbah:quelques} and by
Ginzburg \cite{ginzburg:characteristic}. Our methods and results,
using `shadows' of characteristic cycles in the cotangent bundle of
$X$, may be of independent interest.

For flag manifolds, we apply our methods to relate CSM classes of
Schubert cells, characteristic cycles of Verma $\caD$-modules,
and Maulik and Okounkov stable envelopes for the cotangent bundle of
flag manifolds \cite{maulik.okounkov:quantum}.  This adds a
`Lagrangian perspective' to previous connections between CSM classes
and stable envelopes found by Rim{\'a}nyi and Varchenko
\cite{rimanyi.varchenko:csm} (see also Su's thesis \cite{su:thesis}).
We give a more precise account of our results next.

\subsection{Statement of results} 
The first part of the paper (\S\S\ref{s:shadows}-\ref{s:cc}) consists of a 
general discussion of Ohmoto's torus equivariant CSM classes from 
the point of view of
Lagrangian cycles in the cotangent bundle of a smooth complex
algebraic variety $X$. We extend the formalism of {\em shadows\/} of
characteristic cycles developed in~\cite{aluffi:shadows} to build a
dictionary between $\C^*$-equivariant classes in a vector bundle $E$
endowed with a $\C^*$-action by fiberwise dilation and {\em
  homogenizations\/} of shadows of corresponding classes in the
projective completion of $E$ (\Cref{prop:shapb}).~The
homogenization of a class 
$\alpha=\sum_{i=0}^n \alpha_i$ with
$\alpha_i\in H_{2i}(X)$, with respect to the character $\chi$ of $\C^*$, is the class
\[
\alpha^\chi:=\alpha_0+ \chi \alpha_1 +\dots + \chi^n \alpha_n \in H^{\C^*}_0(X)\quad.
\]
Here the character $\chi$ is identified with 
$c_1^{\C^*}(\C_\chi) \in H^2_{\C^*}(\pt)$, the $\C^*$-equivariant first 
Chern class of the $\C^*$-module $\C_\chi$ given by the dilation on~$\C$ 
with character $\chi$; see \S \ref{ss:vs} below. Let $\hbar^{\pm 1}$ be the 
characters defined by $z \mapsto z^{\pm 1}$.
Note that 
$H^{\C^*}_*(X)\cong H_*(X)[\hbar]$ since $\C^*$ acts trivially on $X$.

We extend the formalism further to smooth
varieties $X$ endowed with the action of a torus $T$. This notion
allows us to define a morphism from the group of $T$-equivariant
conical Lagrangian cycles in the cotangent bundle of~$X$ to
$H_*^T(X)$.  One of the main results is the following
(\Cref{thm:zeropb}):
\begin{theorem}\label{theorem 0} 
  Let $X$ be a smooth complex algebraic variety, with a $T$-action. Consider
  the $\C^*$-action dilating the cotangent fibers with character
  $\hbar^{-1}$ on $T^*(X)$.
  Let $\iota: X \to  T^*(X)$ be the zero
  section. Then
\[ 
\iota^*[\CC(\varphi)]_{T \times \C^*} = \csT(\varphi)^\hbar 
\in H_0^{T \times \C^*} (X) \/. 
\]
\end{theorem}
Here $\csT(-)^\hbar$ is the homogenization of Ohmoto's equivariant
MacPherson natural transformation, and $\CC(\varphi)$ is the characteristic
cycle of a constructible function $\varphi$; see \S\ref{ss:scsm}.
\Cref{theorem 0} generalizes to the equivariant case results of
Ginzburg~\cite[Appendix]{ginzburg:characteristic} and
Sabbah~\cite{sabbah:quelques}. In fact, even in the non-equivariant
case, the theorem gives a transparent interpretation, as the pull-back
via the zero section, of Ginzburg's analogous map from {\em loc.~cit.}
(as explained  in~\S\ref{ss:Ginz}).
Our proof, based on the formalism of shadows, is rather elementary and
it avoids the use of equivariant K-theory and the equivariant
Riemann-Roch transformation in \cite{ginzburg:characteristic}; at
the same time it has a natural equivariant extension.
A reader who is not familiar with the work of 
Ohmoto~\cite{ohmoto:eqcsm} 
(or Ginzburg~\cite[Appendix]{ginzburg:characteristic} in the non-equivariant case) 
can take~\Cref{theorem 0} as the starting point for a definition of the
(non-)equivariant Chern class transformation $\csT$ (resp., $c_*$).

In intersection theory, the pull-back via the zero section is closely
related to a Segre operator; cf.~\cite[\S3.3]{fulton:IT}. We extend
this relation equivariantly, using (equivariant) shadows. The
resulting identity is not only used to prove \Cref{theorem 0}, but it
also informs the rest of this paper. In order to state it, we first
define the Segre operator.  Consider the following diagram:
\[
  \xymatrix{ T^*(X) \ar@{^{(}->}[r] \ar@{->}[rd]^{\pi} & \Pbb(T^*(X)
    \oplus \one) \ar@{->}[d]^{\overline{q}}\\ & X \ar@/^/[lu]^\iota}
\]
Here $\Pbb(T^*(X) \oplus \one)$ is the completion of the cotangent
bundle, $\pi, \overline{q}$, the natural projections, and $\iota$ is the
inclusion of the zero section. As before, consider a $\C^*$-action on
$T^*(X)$ acting on the cotangent fibers with character~$\hbar^{-1}$. 
Let $C \subseteq T^*(X)$ be an equivariant cycle in
$T^*(X)$ and let $\overline{C} \subseteq \Pbb(T^*(X) \oplus \one)$ be its
Zariski closure.  Associated with this data, we consider the Segre operator
given by
\[ 
  \Segre^T(C):= \overline{q}_* \Bigl(
  \frac{[\overline{C}]} {c^T(\cO(-1))}\Bigr) =
  \overline{q}_*(\sum_{j \ge 0} c_1^T(\cO(1))^j \cap
  [\overline{C}])
\] 
(as well as its non-equivariant counterpart $\Segre(C)$).
Here $c^T$ denotes the total torus equivariant Chern class of a
$T$-equivariant vector bundle.  In non-equivariant homology, the
Chern classes $c_i$, $i\ge 1$ are nilpotent, therefore this Segre
operator has values in
homology. In the equivariant context the operator has values in the
completion~$\hat{H}^T_*(X)$. 
With this notation, we prove in
\Cref{prop:shadows} and \Cref{cor:segree} the following identity.

\begin{theorem}\label{thm:segre-identity}
  Let $\varphi$ be a $T$-invariant constructible function and
  $\CC(\varphi)$ its characteristic cycle. Then
\begin{equation*}\label{E:intro-key}
  c^T(T^*(X)) \cap \Segre^T(\CC(\varphi)) = \Shadow^T(\CC(\varphi)) 
  = \chc^T_*(\varphi) \in H_*^T(X) \/,
\end{equation*}
where $\chc^T_*(\varphi)$ denotes a `signed' version of Ohmoto's
natural transformation; see~\eqref{E:checkc}.
\end{theorem}

The second part of the paper (\S\S\ref{s:prelfm}-\ref{ss:csmvsstable}) 
applies this formalism to the study of
CSM classes of Schubert cells in (generalized) flag manifolds $X=G/B$,
where $G$ is a complex, semisimple Lie group and $B$ is a Borel
subgroup. Let $T \subseteq B$ be the maximal torus, and $W$ the
associated Weyl group. Let $R^+$~denote the set of positive roots
associated to $(G,B)$. Denote by $X(w)^\circ:= B w B/B$ the Schubert
cell for the Weyl group element $w \in W$. Further, let
$\cM_w$ be the Verma $\caD_{X}$-module determined by
the Verma module of highest weight $-w(\rho) - \rho$, where $\rho$ is
half the sum of positive roots. Denote by $\Char(\cM_w)$ the
characteristic cycle of the
holonomic $\caD_{X}$-module~$\cM_w$; see~\S\ref{CSMVerma}.
This is a conic $T \times \C^*$-stable Lagrangian
cycle in $T^*(X)$.

The relevance of the Verma module comes from the proof of the
Kazhdan-Lusztig conjectures by Brylinski-Kashiwara
\cite{brylinski.kashiwara:KL} and Beilinson-Bernstein
\cite{beilinson.bernstein:localisation}. There it was shown that
$\Char (\cM_w)$ is (up to a sign) equal to the the
characteristic cycle $\CC(\one_{X(w)^\circ})$.  The main result of the
paper involves a formula for the CSM classes in terms of the Segre
operator applied to characteristic cycle of the Verma module
(\Cref{thm:csmsegre}):

\begin{theorem}\label{thm:main}
  Let $w \in W$ be an element in the Weyl group. Then the following
  equality holds:
  \[
    \csmT(X(w)^\circ) = \Bigl( \prod_{\alpha \in R^+} (1 + \alpha)
    \Bigr)\, \Segre^T(\Char(\cM_w)) \in \hat{H}^T_*(G/B)\/.
  \]
  In particular, in non-equivariant homology,
  \[
    \csm(X(w)^\circ) = \Segre (\Char(\cM_w)) \in H_*(G/B)\/.
  \]
\end{theorem}

In the non-equivariant case the Segre class is
manifestly effective, and this implies the positivity of CSM classes. The
proof of this consequence (but {\em not\/} of \Cref{thm:main}!) is
independent of other facts in the paper, and will be given next. Let
$[X(w)] \in H_*(X)$ be the fundamental class of the Schubert variety
$X(w):= \overline{X(w)^\circ}$. The set of fundamental classes
$\{ [X(w)] \}_{w \in W}$ forms a basis of the $\Z$-module $H_*(X)$.

\begin{corol}[Positivity of CSM classes]\label{cor:intropos}
  Let $w \in W$. Then the non-equivariant CSM class $\csm(X(w)^\circ)$
  is effective, i.e., in the Schubert expansion
  \[
    \csm (X(w)^\circ) = \sum_{v \le w} c(v;w) [X(v)] \in H_*(X)\/,
  \]
  the coefficients $c(v;w)$ are non-negative. Further, let
$P \subseteq G$ be any parabolic subgroup. Then the CSM classes of
Schubert cells in $H_*(G/P)$ are effective, i.e., its coefficients in the 
corresponding Schubert expansion are nonnegative.
\end{corol}

\begin{proof}
  Consider first $X=G/B$. The characteristic cycle of
  the Verma $\caD_{X}$-module $\Char(\cM_w)$ in $T^*(X)$ is effective
  \cite[p.~60]{HTT},
  and so is its closure in $\Pbb(T^*(X) \oplus \one)$.~(The Verma module is also 
  holonomic, and in this case effectivity is explicit from
  \cite[p.~119]{HTT}, or from 
  \cite[Remark~6.0.4 on p.~389]{schurmann:book} 
  in terms of perverse
  sheaves.)
  The
  tautological line bundle $\cO_{T^*(X) \oplus \one}(1)$ is globally
  generated, i.e., it is a quotient of a trivial bundle. Indeed,
  $\cO_{T^*(X) \oplus \one}(1)$, is a quotient of $T(X) \oplus \one$,
  and by homogeneity of $X$, $T(X)$ is globally generated. Then
  $c_1 (\cO_{T^*(G/B) \oplus \one}(1)) \cap [\overline{C}]$ is
  effective, for any effective cycle $\overline{C}$ \cite[Ex.~12.1.7
  (a)]{fulton:IT}. The result follows from Theorem \ref{thm:main}.
  The effectivity of CSM classes in $G/B$ implies the effectivity in 
  $G/P$ by~\cite[Proposition~3.5]{aluffi.mihalcea:eqcsm}.
\end{proof}

This generalizes the positivity of CSM classes of Schubert cells for
Grassmann manifolds, which was proved in several cases by Aluffi and
Mihalcea \cite{aluffi.mihalcea:csm,mihalcea:binomial}, Jones
\cite{jones:csm}, and Stryker \cite{MR2949825}, and for any Schubert
cell in any Grassmann manifold by J.~Huh \cite{huh:positivity}. In
 fact, Huh proved a stronger version of the positivity conjecture, asserting
that each homogeneous component of the CSM class
of a Schubert cell is equal to the fundamental class of a non-empty subvariety
inside the corresponding Schubert variety.
Huh's proof used that Schubert varieties in Grassmann
manifolds can be desingularized by varieties with finitely many Borel
orbits. Unfortunately, the known desingularizations of Schubert
varieties in arbitrary flag manifolds do not satisfy this property in
general. Seung Jin Lee \cite[Theorem~1.1]{sj:chern} proved that the
positivity of CSM classes for type A flag manifolds is {\em implied}
by a positivity conjecture in a certain subalgebra of Fomin-Kirillov
algebra \cite{fomin.kirillov:quadratic} generated by Dunkl
elements. Thus CSM positivity can be also regarded as evidence for the
Fomin-Kirillov conjecture. Further consequences of the positivity of
CSM classes are the positivity of certain `Kazhdan-Lusztig' classes,
the positivity of Chern-Mather classes for the Schubert varieties in
the minuscule Grassmannians of classical Lie type (both discussed in
\S \ref{sec:KLclasses}), and the fact that the {\em Segre-MacPherson
  classes} of Schubert cells in $G/B$ are Schubert alternating; see
\Cref{thm:dual}.

The proof of \Cref{thm:main} is based on \Cref{thm:segre-identity} and
two orthogonality properties.  Let $Y(w)^\circ = B^-wB/B \subseteq G/B$
be the Schubert cell determined by the opposite Borel subgroup
$B^-$. Then $G/B$ has two transversal Whitney stratifications:
$G/B = \bigsqcup_w X(w)^\circ = \bigsqcup_w Y(w)^\circ$.  By
\Cref{thm:gporth},
\begin{equation}\label{E:transv-csm-schubert}
\left\langle \csmT(X(u)^\circ), \frac{\csmT(Y(v)^\circ)}{c^{T}(T(G/B))} 
\right\rangle = \delta_{u,v} \/. 
\end{equation}
Non-equivariantly, this was proved by Sch{\"u}rmann
\cite{schurmann:transversality} for complex varieties with
transversal Whitney stratifications; see \Cref{thm:intersection}. In an
Appendix to this paper (\S\ref{s:app}) we include an outline of that proof, including
the details required to extend it equivariantly.  A similar
orthogonality holds for Maulik and Okounkov stable envelopes
\cite{maulik.okounkov:quantum}, and it was used to
prove~\eqref{E:transv-csm-schubert} in an earlier ar$\chi$iv version of this
paper.  In the language of {\em weight functions,\/} equivalent
statements were obtained in
\cite{RTV:Yangian,RTV:Kstable,RTV:partial,GRTV:quantum} for the type
A flag manifolds.

The second orthogonality is derived from the fact, proved in
\cite{aluffi.mihalcea:eqcsm}, that the CSM classes may be calculated
recursively via Demazure-Lusztig (DL) operators.  These generate the
degenerate Hecke algebra of $W$ \cite{LLT:twisted,
  LLT:flagYB,ginzburg:methods}; see \S \ref{sec:DLops}. The key fact
is that the classes of the adjoint DL operators generate Poincar{\'e}
dual classes, giving the following (cf.~\Cref{thm:orthogonality}):
\begin{equation}\label{E:hecke-orth} 
\langle \csmT(X(u)^\circ), \csmTv(Y(v)^\circ) \rangle 
= \delta_{u,v} \prod_{\alpha \in R^+} (1+ \alpha) \/. 
\end{equation} 
Here $\csmTv(Y(v)^\circ)$ is a `signed' version of the CSM class,
where the signs of the homogeneous components are changed according to
complex codimension. We call this the `Hecke orthogonality'. Combining
identities~\eqref{E:transv-csm-schubert} and~\eqref{E:hecke-orth} yields the
remarkable identity
\[
  \csmTv(Y(v)^\circ) = \prod_{\alpha \in R^+} (1+ \alpha)
  \frac{\csmT(Y(v)^\circ)}{c^{T}(T(G/B))} \/.
\]
By \Cref{thm:segre-identity}, the right-hand side of this identity is
essentially equal to the Segre operator $\Segre^T(\Char(\cM_w))$,
except that {\em we need to change signs of homogeneous
  components}. Then the formula from \Cref{thm:main} will follow; see
\Cref{thm:csmsegre}.

As further applications, observe that \Cref{theorem 0}, applied to the
indicator function of the Schubert cell $\varphi = \one_{X(w)^\circ}$,
implies that
\[ 
\csmTh(X(w)^\circ) = (-1)^{\ell(w)} \iota^*[\Char(\cM_w)]\/. 
\] 
The Verma characteristic cycle in the right-hand side equals Maulik
and Okounkov's {\em stable envelope} $\stab_+(w)$. This is stated
without proof in \cite[Remark~3.5.3, p.~69]{maulik.okounkov:quantum},
and for completeness we sketch an argument in {Lem\-ma}~\ref{lemma:stab}.
Combining the two facts one immediately obtains the following corollary
(cf.~Corollary~\ref{cor:csmstab}):

\begin{corol}\label{cor0}
  Let $w \in W$. Then
  $\iota^*(\stab_+(w)) = (-1)^{\dim X} \csmTh(X(w)^\circ)$ as
    elements in $H_0^{T \times \C^*}(X)$.
\end{corol}

This equality was proved earlier by Rim{\'a}nyi and Varchenko
\cite[Theorem~8.1 and Remark~8.2]{rimanyi.varchenko:csm} (see also
\cite{su:thesis}), using interpolation properties of CSM classes
stemming from Weber's work \cite{Weber} and the defining localization
properties of the stable envelopes; cf.~\S\ref{ss:csmvsstable}.
 For a K theoretic version of these results, see e.g., \cite{AMSS:motivic}.
Once this is established, we use the dictionary between CSM
classes and stable envelopes to prove various results about the
former, notably a localization formula (\Cref{cor:localizations}) and
a Chevalley formula (\Cref{thm:che for csm}), which might be of
independent interest.

\subsection{Conventions and notation}
We work over $\C$. {\em Varieties\/} are reduced and irreducible. 
{\em Subvarieties\/} are assumed to be closed. An {\em irreducible cycle\/}
is the cycle of a subvariety. We will frequently use the following notation;
we indicate here where the notation is defined.\smallskip

$\alpha^\chi$, homogenization of a homology class $\alpha$ by a character $\chi$ \dotfill 
\S\ref{ss:vs} \eqref{eq:homogdef}\smallskip

$\Shadow^T$, equivariant shadow \dotfill 
\S\ref{sec:eqshadows} \eqref{eq:defeqsha}\smallskip

$\csmT$, equivariant Chern-Schwartz-MacPherson (CSM) class \dotfill 
Definition~\ref{def:CSM}\smallskip

$\cMaT$, equivariant Chern-Mather class \dotfill 
Definition~\ref{def:CSM}\smallskip

 $\chc_*$, $\chc_*^T$, signed (Ohmoto-)MacPherson natural transformations \dotfill 
\S\ref{ss:bn}, \eqref{E:checkc}\smallskip

$\csmTv$, $\cMaTv$, dual equivariant CSM and Mather classes \dotfill 
\S\ref{ss:bn}, \eqref{E:defdual}\smallskip

$\csmTh$, $\csmThv$, homogenized equivariant CSM classes/dual CSM classes \dotfill 
\S\ref{CSMVerma}, \S\ref{ss:Csealf}\smallskip

$\cEu$, signed Euler obstruction \dotfill \S\ref{ss:scsm}\smallskip

$\hbar$, identity character of $\C^*$ \dotfill \S\ref{ss:ECSM}

$\Segre^T$, equivariant Segre operator \dotfill
\S\ref{sec:eqshadows}, \eqref{E:defsegre}\smallskip

 $\ssm$, $\ssmT$, (equivariant) Segre-MacPherson class \dotfill
\S\ref{ss:bn}, \eqref{E:segre-new}\smallskip

$X(w)^\circ$, Schubert cell  \dotfill
\S\ref{s:prelfm}\smallskip

$Y(w)^\circ$, opposite Schubert cell  \dotfill
\S\ref{s:prelfm}\smallskip

$\cT_i, \cT_i^\vee$, Demazure--Lusztig operators \dotfill
\S\ref{sec:basic-flags}, \eqref{E:DLops}\smallskip

$\mathcal{M}_w$, Verma module \dotfill \ref{CSMVerma}\smallskip

$\mathrm{KL}(w)$, Kazhdan-Lusztig classes \dotfill \ref{def:KL}\smallskip

$\stab_+(w)$, stable envelopes \dotfill
\S\ref{ss:csmvsstable}, \Cref{definition of stable basis} 

\bigskip

{\em Acknowledgments.} P.~Aluffi is grateful to Caltech for
hospitality while part of this work was carried out. L.~C. Mihalcea
wishes to thank D.~Anderson, A.~Knutson, S.~J.~Lee, M.~Shimozono,
D.~Orr and R.~ Rim{\'a}nyi for stimulating discussions; and further
A.~Knutson for being a catalyst to this collaboration. C.~Su wishes to
thank his advisor A.~Okounkov for encouragements and useful
discussions. Part of this work was carried while at the $2015$ PCMI
Summer session in Park City; L.~C.~Mihalcea and C.~Su are both
grateful for support and for providing an environment conducive to
research. Finally, we are indebted to two anonymous referees, whose
careful reading and many valuable suggestions helped us 
improve substantially the structure of this paper.


\section{Shadows and equivariant (co)homology}\label{s:shadows} 

\subsection{Equivariant (co)homology}\label{ss:ECSM}
In this paper we work in the complex algebraic context and
utilize $H_*(X)$, the Borel-Moore homology group of~$X$, 
and $H^*(X)$, the cohomology ring. The reader so inclined may
  use Chow (co)homology instead: there is a homology degree doubling cycle map
  between the Chow and Borel-Moore groups, and our constructions are
  compatible with this map. This map is an isomorphism in some important
  situations, such as the complex flag manifolds, studied later in
  this paper. We refer to \cite[\S19.1]{fulton:IT} and \cite[\S2.6]{ginzburg:methods} 
  for more details about Borel-Moore homology
  and its relation to Chow groups.  In case we speak of (co)dimension
we always assume that our spaces are pure dimensional; in addition,
by (co)dimension
we will mean the {\em complex} (co)dimension. Thus any
subvariety of (complex) dimension $k$ has a fundamental
class $[Y] \in H_{2k}(X)$.  Whenever $X$ is smooth, we can and will
identify the Borel-Moore homology and cohomology via Poincar{\'e}
duality.

Let $T$ be a torus and let $X$ be a variety with a $T$-action. Then
the equivariant cohomology $H^*_T(X)$ is the ordinary cohomology of
the Borel mixing space $X_T:= (ET \times X)/T$, where $ET$ is the
universal $T$-bundle and $T$ acts by $t \cdot (e,x) = 
(e t^{-1}, t x)$. It is an algebra over $H^*_T(\pt)$, the polynomial ring 
$\Sym_{\Z} \mathfrak{X}(T)$ 
in the character group $\mathfrak{X}(T)$; see e.g.,~\cite{MR2976939, AndersonFulton} or~\cite[\S11.3.5]{kumar:KMbook}. 
 In this paper we utilize the $T$-equivariant version of Borel Moore homology theory, and 
one may alternatively work with the $T$-equivariant Chow
groups. The two theories are related by an equivariant cycle map.
The space $ET \simeq (\C^\infty \setminus 0)^{\mathrm{rank}~T}$ is infinite dimensional, 
and the latter theories utilize finite dimensional approximations 
$U \simeq (\C^N \setminus 0)^{\mathrm{rank}~T} \subset ET$, giving approximations 
$U \times_T X$ of the mixing space $X_T$; see 
\cite{edidin.graham:eqchow}.\begin{footnote}{In general, the  algebraic 
approximation $(U\times X)/T$ of the Borel mixing space $X_T$
is only a
    separated algebraic space, but if $X$ is a quasi-projective
    scheme, then  $(U\times X)/T$ is again quasi projective; see 
    \cite[\S2.2]{ohmoto:eqcsm} or \cite{ohmoto:note}.}
\end{footnote}
Every $k$-dimensional subvariety
$Y\subseteq X$ that is stable under the $T$ action determines
an equivariant fundamental class $[Y]_T$ in
$H_{2k}^T(X)$. As in the non-equivariant case, whenever
  $X$ is smooth, we will identify $H_*^T(X)$ and $H^*_T(X)$.  In
  particular, when $X=\pt$, the identification sends $a \in H^*_T(\pt)$
  to $a \cap [\pt]_T$.
We address the reader to~\cite{MR2976939, AndersonFulton, knutson:noncomplex}, or
\cite{ohmoto:eqcsm} for basic facts on equivariant cohomology and
homology. 
 Equivariant vector bundles have equivariant Chern classes
 {$c^T_j(-)\in H^{2j}_T(X)$, such that $c_j^T(E)\cap -$ is an operator
  $H^T_i(X)\to H^T_{i-2j}(X)$; see \cite[\S1.3]{MR2976939},
\cite[\S2.4]{edidin.graham:eqchow}.
In terms of the Borel construction, $c_j^T(E)$ corresponds to 
$c_j(E_T)\in H^{2j}(X_T)$ for the vector bundle $E_T\to X_T$.
 Then the above identification $H^*_T(\pt)\cong \Sym_{\Z} \mathfrak{X}(T)$ is induced
by the isomorphism $\mathfrak{X}(T)\cong H^2_T(\pt)$, which associates 
to a  character $\chi\in \mathfrak{X}(T)$ the first equivariant Chern class $c^T_1(\C_\chi)
\in H^2_T(\pt)$ of the $T$-module $\C_\chi$  given by the dilation on $\C$ with character $\chi$
(regarded as an equivariant line bundle over a point).

Let $\pi: E \to X$ be a vector bundle of rank $e+1$ on $X$. We
consider the action of $\C^*$ on $E$ by fiberwise dilation with
character $\chi$, and denote by $E^\chi$ the vector bundle $E$ endowed
with this $\C^*$-action. 
Equivalently, $E^\chi\cong E\otimes \C_\chi$, where $\C^*$ acts 
trivially on $E$ and $X$ and via $\chi$ on $ \C_{\chi}$.
The natural projection $\pi: E^\chi \to X$ is
equivariant, where $\C^*$ acts trivially on $X$.~We may take 
$E\C^* = \C^\infty \setminus 0$; since the action of
$\C^*$ on $X$ is trivial, the Borel mixing space $X_{\C^*}$ is
isomorphic to $B\C^*\times X=\Pbb^\infty\times X$. Here and in the
following, we will denote by $\Pbb^\infty$ any approximation~$\Pbb^N$
with $N\gg 0$ sufficiently large; see e.g., \cite[\S 1.2]{MR2976939}.
We will give the results in the ordinary and equivariant
Borel-Moore homology groups.
Since $X_{\C^*}\cong \Pbb^\infty \times X$,
\[
H_*^{\C^*}(X_{\C^*}) \cong H_*(X) [\hbar]\quad,
\]
 where $\hbar:=c_1(\cO_{\Pbb^\infty}(-1)) \in H^2(\Pbb^\infty)=H^2_{\C^*}(\pt)$ 
corresponds to the equivariant first Chern class 
$\hbar:= c_1^{\C^*}(\C_{\chi_1}) \in H^2_{\C^*}(\pt)$ of the $\C^*$-module $\C_{\chi_1}$ with $\chi_1$ 
the character $z\mapsto z^1$.
Next, denote by
\[
\xymatrix{
\rho: X_{\C^*} = \Pbb^\infty \times X \ar[r] & X
}
\]
the projection. If 
$\chi$ is the character $\chi_a$ given by $z\mapsto z^a$, a standard
computation shows that the mixing space $E^\chi_{\C^*}$, along with
the natural projection to $X_{\C^*}$, is isomorphic to the vector
bundle
$\pi^{\chi}: \rho^* E \otimes \cO_{\Pbb^\infty}(-a) \to X_{\C^*}$.  We
have the diagram
\begin{equation}\label{E:mdiagram}
\vcenter{\xymatrix{
E^\chi_{\C^*} = \rho^* E\otimes \cO_{\Pbb^\infty}(-a) \ar[d]_{\pi^\chi} & 
E^\chi \ar[d]^\pi \\
X_{\C^*} = \Pbb^\infty\times X \ar[r]^-\rho & X
}}
\end{equation}

\begin{lemma}\label{lem:eqcohE}
  The projection $\pi$ induces by flat pull-back a
  codimension-preserving isomorphism
  $\pi^*: H_*^{\C^*} (X) \overset\sim\longrightarrow H_*^{\C^*}
  (E^\chi)$.

  The embedding $\iota: X \to E$ of the zero section induces a
  codimension-preserving isomorphism
\[
\xymatrix{
\iota^*: H_*^{\C^*} (E^\chi) \ar[r]^-\sim & H_*^{\C^*} (X)\cong H_*(X)[\hbar]\quad \/, 
}
\]
satisfying $\iota^* = (\pi^*)^{-1}$.
\end{lemma}
\begin{proof} 
  This follows from \cite[Theorem~3.3(a)]{fulton:IT} applied to the
  projection $\pi^\chi: E^\chi_{\C^*} \to X_{\C^*}$.
\end{proof}

\subsection{Shadows: definition and basic properties}\label{ss:shadows} 
Let $\pi: E \to X$ be a vector bundle of rank $e+1$ on a variety $X$, and
consider the projective bundle of lines $q:\Pbb(E)\to X$. 
Let $\zeta=c_1(\cO_E(1))$. Then $H_*(\Pbb(E))$ is a direct sum of $e+1$
copies of $H_*(X)$: every class $\aP$ of codimension $k$ in $H_*(\Pbb(E))$ 
may be written as
\begin{equation}\label{eq:strpr}
\aP = \sum_{j=0}^e \zeta^j q^*(\aX^{k-j}) 
\end{equation}
for uniquely defined classes $\aX^{k-j}$ of codimension $k-j$ in $H_*(X)$
(cf.~\cite[Theorem~3.3(b)]{fulton:IT}). In fact, we have the relation
\[
\zeta^{e+1} + \zeta^e q^*c_1(E) + \cdots + \zeta q^*c_e(E) +q^*c_{e+1}(E)=0
\]
in $H_*(\Pbb(E))$ (\cite[Remark~3.2.4]{fulton:IT}).

Following~\cite{aluffi:shadows}, we call the (non-homogeneous) class
\[
\Shadow_E(\aP):= \aX^{k-e} + \aX^{k-e+1} + \cdots + \aX^k \in H_*(X)
\]
the {\em shadow\/} of $\aP$ in $X$. By~\eqref{eq:strpr}, a homogeneous
class $\aP\in H_*(\Pbb(E))$ may be reconstructed from its shadow and its
codimension $k$. We will omit the subscript $E$ from the notation if
the ambient projective bundle is understood from the context.  The
following lemma is useful in relating shadows of classes in different
projective bundles, with $c(-)\cap$ the total Chern class operator.

\begin{lemma}\label{lem:lem42}
For a class $\aP$ in
$H_*(\Pbb(E))$  
the following hold in~$H_*(X)$.
\begin{itemize}
\item[(i)] The shadow $\Shadow_E(\aP)$ equals
\[
c(E)\cap q_* (c(\cO_E(-1))^{-1}\cap \aP)
=c(E)\cap q_*\sum_{j\ge 0} c_1(\cO_E(1))^j\cap \aP \quad;
\]
\item[(ii)]
If $F$ is a subbundle of $E$, and 
$\aP\in H_*(\Pbb(F))\hookrightarrow H_*(\Pbb(E))$, then
\[
\Shadow_E(\aP)=c(E/F)\cap \Shadow_F(\aP)\quad;
\]
\item[(iii)]
  If $\Shadow_E(\aP)=\sum_{j=0}^e \aX^{k-j}$, and $L$ is a line bundle
  on $X$, then
\[
\Shadow_{E\otimes L}(\aP)=\sum_{j=0}^e c(L)^j\cap \aX^{k-j}\quad;
\]
\item[(iv)]
  Further, let $E \to F$ be a surjection of bundles, with kernel $K$,
  and let $C_E$ be a cycle in $\Pbb(E)$ disjoint from $\Pbb(K)$.  Let
  $C_F$ be the cycle in $\Pbb(F)$ obtained by pushing forward~$C_E$.
  Then
\[
\Shadow_E([C_E])=c(K)\cap \Shadow_F([C_F])\quad.
\]
\end{itemize}
\end{lemma}

\begin{proof}
  Part~(i) is~\cite[Lemma~4.2]{aluffi:shadows}. Part~(ii) follows
  immediately from (i). Part~(iii) is a straightforward computation,
  which we leave to the reader. For part~(iv), let
  $\underline q:\Pbb(F) \to X$ be the projection.  The given
  surjection $E\to F$ induces a rational map
  $\Pbb(E)\dashrightarrow \Pbb(F)$, which is resolved by blowing up
  along $\Pbb(K)$; let $\nu:\Til \Pbb\to \Pbb(E)$ be this blow-up, and
  let $\underline\nu: \Til\Pbb \to \Pbb(F)$ be the induced morphism:
\[
\xymatrix@C=10pt@R=10pt{
& {\Til \Pbb} \ar[dl]_{\nu} \ar[dr]^{\underline\nu} \\
\Pbb(E) \ar@{-->}[rr] \ar[dr]_q & & \Pbb(F) \ar[dl]^{\underline q} \\
& X
}
\]
Since $C_E$ is disjoint from $\Pbb(K)$ and $\nu$ is an isomorphism
over $\Pbb(E)\smallsetminus \Pbb(K)$, the cycle $C_E$ determines a
cycle $\Til C_E$ in $\Til\Pbb$, disjoint from the exceptional divisor,
such that $[C_E]=\nu_*[\Til C_E]$ and
$[C_F] = \underline\nu_*[\Til C_E]$. By part~(i) and the projection
formula,
\[
\Shadow_E([C_E]) = c(E)\cap q_*\nu_*(c(\nu^*\cO_E(-1))^{-1}\cap [\Til C_E])\quad.
\]
Now note that $\nu^*\cO_E(-1)$ and $\underline\nu^*\cO_F(-1)$ differ
by a term supported on the exceptional divisor in $\Til\Pbb$, hence
they agree on $\Til C_E$. Therefore
\begin{align*}
\Shadow_E([C_E]) 
&= c(E)\cap \underline q_*\underline \nu_*(c(\underline\nu^*\cO_F(-1))^{-1}\cap [\Til C_E]) \\
&= c(E)\cap \underline q_*(c(\cO_F(-1))^{-1}\cap [C_F]) \\
&= c(K)\cap \Shadow_F([C_F])
\end{align*}
again by the projection formula and part~(i).
\end{proof}

The formula in part (i) may be expressed concisely in terms of a
`Segre class operator', which we will introduce (in the equivariant setting)
in~\S\ref{sec:eqshadows}.

Shadows are compatible with the operation of taking a cone.  More
precisely, let $\one$ denote the trivial rank-$1$ line bundle on $X$,
and consider the projective completion $\Pbb(E\oplus \one)$; $E$~may
be identified with the complement of $\Pbb(E\oplus 0)$ in
$\Pbb(E\oplus \one)$. Consider a $\C^*$-action on $E$ by fiberwise
dilation, and the trivial $\C^*$-action on $\one$.  This induces a
$\C^*$-action on $\Pbb(E \oplus \one)$ such that the inclusion
$E \subseteq \Pbb(E \oplus \one)$ is $\C^*$-equivariant, and the trivial
action on $\Pbb(E)= \Pbb(E \oplus 0)$.  A class $\aP$ in
$H_*(\Pbb(E))$ determines a $\C^*$-stable class $\aPp$ in
$H_*(\Pbb(E\oplus \one))$, obtained by taking the cone with vertex along
the zero-section $X= \Pbb(0\oplus \one)$.

\begin{lemma}\label{lem:eqsha}
$\Shadow_E(\aP) = \Shadow_{E\oplus \one}(\aPp)$.
\end{lemma}

\begin{proof}
  This follows from~\Cref{lem:lem42}~{(i)}. Indeed
  $c(E\oplus \one) = c(E)$, and $\Pbb(E)\cong \Pbb(E\oplus 0)$
  represents $c_1(\cO_{E\oplus \one}(1))$, so that
\[
\sum_{j\ge 1} c_1(\cO_{E\oplus\one}(1))^j \cap \aPp
=\sum_{j\ge 0} c_1(\cO_E(1))^j \cap \aP \/;
\]
(the remaining term vanishes in the push-forward for dimensional
reasons).
\end{proof}

\begin{remark}\label{rem:zs}
Not all $\C^*$-stable classes in 
$H_*(\Pbb(E\oplus \one))$
are obtained from classes in 
$H_*(\Pbb(E))$ as in \Cref{lem:eqsha}.  For instance, the class of the
zero section $X = \Pbb(0 \oplus \one)$ is $\C^*$-fixed and not of
this form. For any subvariety
$V\subseteq X=\Pbb(0\oplus \one)\subseteq \Pbb(E\oplus \one)$,
\[
\Shadow_{E\oplus \one}([V]) = c(E)\cap [V]
\]
by \Cref{lem:lem42}~(i), since $\cO_{E\oplus \one}(-1)$ is trivial along the
zero-section.
\qede\end{remark}

Denote by $i:E \to \Pbb(E\oplus \one)$ the embedding of $E$ as the
complement of $\Pbb(E\oplus 0)$, and by
$\overline q:\Pbb(E\oplus \one) \to X$ the projection.

\begin{lemma}\label{lem:restr}
If $\aP\in H_*(\Pbb(E\oplus \one))$ has codimension
$k$, then $i^*(\aP)=\pi^*(\aX^k)$ where
$\aX^k$ is the component of $\Shadow_{E\oplus
  \one}(\aP)$ of codimension~$k$.
\end{lemma}

\begin{proof}
  Indeed, $\aP=\sum_{j=0}^{e+1} \eta^j \overline
  q^*(\aX^{k-j})$, where $\Shadow(\aP)=\sum_{j=0}^{e+1}
  \aX^j$ and $\eta=c_1(\cO_{E\oplus\one}(1))$.  Since $\Pbb(E\oplus
  0)$ represents $\eta$ and is disjoint from $E$,
  $i^*(\eta)=0$.  Therefore $i^*(\aP)=i^*\overline q^*
  (\aX^k)=\pi^*(\aX^k)$ as stated.
\end{proof}

\subsection{Homogenization of shadows and $\C^*$-equivariant
  homology}\label{ss:vs}
Consider a class
$\alpha= \sum_{j=0}^\ell \alpha^j \in H_*(X)$, where
$\alpha^j$ denotes the homogeneous component of
$\alpha$ of codimension $j$ in $X$.  The choice of a codimension $k
\ge \ell$ and of a character $\chi$  of $\C^*$
determine the homogeneous class
\begin{equation}\label{eq:homogdef}
\alpha^\chi := \sum_{j=0}^\ell \chi^{k-j} \alpha^j \in H_{2(\dim X - k)}^{\C^*}(X) \/;
\end{equation}
as usual we identify the character $\chi=\chi_a$ given by  
$z\mapsto z^a$ with its equivariant first Chern class 
$c_1^{\C^*}(\C_\chi) =a\hbar \in H^2_{\C^*}(\pt)$.
~We will call
$\alpha^\chi$ the {\em `($\chi$-){homoge\-nization}'} of degree
$k$ of $\alpha$; the fixed codimension
$k$ will often be clear from the context.
We will use the additive notation for characters in the context of homogenizations,
and the multiplicative notation in the context of group actions. For instance,
$\alpha^{-\hbar} = \sum_{j=0}^\ell (-\hbar)^{k-j}\alpha^j $ is determined by the character
$\hbar^{-1} $ given by $z \mapsto z^{-1}$.
\begin{example}\label{ex:topc}
  A key example is given by the homogenization of the total Chern
  class of the bundle $E^\chi$. If $x_1, \ldots ,
  x_{e+1}$ are the (non-equivariant) Chern roots of
  $E$ then the ($\C^*$-equivariant) Chern roots of
$E^\chi\cong E\otimes \C_\chi$ are
$x_1 +\chi , \ldots , x_{e+1} + \chi$.
It follows that for every subvariety $V\subseteq X$,
\begin{equation}\label{E:homce} 
c_{e+1}^{\C^*}(E^\chi)\cap [V]_{\C^*} = (c(E)\cap [V])^{\chi} \in 
H_{2(\dim V-(e+1))}^{\C^*}(X)
\end{equation}
(note that $[V]$ may be identified with
$[V]_{\C^*}$ since the $\C^*$-action on
$X$ is trivial). I.e., the homogenization of the total Chern class of $E$
is naturally a top equivariant Chern class.  
\qede\end{example}

Now let $C$ be a $\C^*$-stable cycle of codimension $k$ in
$E=E^\chi$. Viewing $E$ as an open subset of $\Pbb(E\oplus \one)$ as
above, the closure of $C$ is a codimension-$k$ cycle $\overline C$ in
$\Pbb(E\oplus \one)$.  The next result compares the class $[C]_{\C^*}$
of $C$ in the equivariant Borel-Moore homology group~$H_*^{\C^*}(E^\chi)$ 
and the class~$[\overline C]$ in the ordinary
Borel-Moore homology group $H_*(\Pbb(E\oplus \one))$.

\begin{prop}\label{prop:shapb}
  Let $C$ be a $\C^*$-stable cycle of codimension $k$ in $E^\chi$,
  as above.  Then
  $[C]_{\C^*}\in H_*^{\C^*}(E^\chi)\overset {\iota^*}\cong
  H_*(X)[\hbar]$ is the $\chi$-homogenization of degree $k$ of the
  shadow of $[\overline C]$:
\[
\iota^*([C]_{\C^*}) = (\Shadow([\overline C]))^\chi\quad.
\]
\end{prop}

\begin{remark}
  In particular, this shows that equivariant fundamental classes of
  subvarieties of a vector bundle $E^\chi$ of rank $e+1$ are
  of the form
\[
  \alpha^k + \chi  \pi^*(\alpha^{k-1}) + \cdots + \chi^{e+1}
  \pi^*(\alpha^{k-e-1})
\]
i.e., combinations of powers 
$\chi^j = \left( c_1^{\C^*}(\C_\chi)\right)^j$
 with $0\le j\le \rk E$. In
  other words, among all equivariant classes, \Cref{prop:shapb}
  distinguishes the fundamental classes of fixed codimension
  equivariant subvarieties in $E$ as those classes determined by no
  more than $\rk E+1$ homogeneous classes in
  $H_*(X)$, a quantity independent of $\dim X$.  It will in fact
follow from the proof that if such a subvariety
is not supported within the zero section of $E$, then
$\alpha^{k-e-1}=0$  (cf.~\eqref{eq:toe}).  
\qede\end{remark}

\begin{proof}[Proof of \Cref{prop:shapb}]
  By linearity we may assume that $C$ is the cycle of a $\C^*$-stable
  subvariety~$V$ of $E$.  First assume that the subvariety is
  contained in the zero-section, so that $C=\iota_*([V])$ for a
  subvariety $V$ of $X$. By the (equivariant) self-intersection
  formula,
\[ 
\iota^*([C]_{\C^*}) = \iota^* ( \iota_* [V]_{\C^*}) 
= c_{e+1}^{\C^*} (E^\chi) \cap [V]_{\C^*} 
= (c(E) \cap [V])^\chi \/, 
\] 
where the last equality follows from~\eqref{E:homce}. ($[V]_{\C^*}$ may
be identified with $[V]$ since the $\C^*$-action on $X$ is trivial.)
On the other hand, $\overline{C}= C$ is also the push-forward of $V$
to the zero-section $X =\Pbb(0\oplus \one)\subseteq \Pbb(E\oplus \one)$.
As in \Cref{rem:zs}, we deduce that
\[
\Shadow([\overline C]) = c(E\oplus \one)\cap [V] = c(E) \cap [V] 
\] 
from which the claim follows.

Next, assume $V$ is {\em not\/} supported on the zero-section of $E$
and has codimension~$k$.  Since $V$ is $\C^*$-stable,
$V$~determines and is determined by a subvariety $\Pbb(V)$ of
$\Pbb(E)$; in this case, $\overline V \subseteq \Pbb(E\oplus\one)$ 
is the cone over $\Pbb(V)$ with vertex along $\Pbb(0\oplus \one)$. Let
$\Shadow_E (\Pbb(V)) = \sum_{j=0}^e \aX^{k-j}$, with $\aX^{k-j}$ of
codimension $(k-j)$. By \Cref{lem:eqsha} this is also the shadow of
$\overline V$, so that
\begin{equation}\label{eq:Shachi}
(\Shadow_{E \oplus \one} ([\overline V]))^\chi = \sum_{j=0}^e \chi^j \aX^{k-j}
\end{equation}
in $H_{2(\dim X-k)}^{\C^*}(X)$.  Denote by
$i^\chi: E^\chi \to \Pbb(E^\chi \oplus \one)$ the open embedding with
complement $\Pbb(E^\chi \oplus 0)$. This is a flat $\C^*$-equivariant
map, therefore $[V]_{\C^*} = (i^\chi)^*[\overline V]_{\C^*}$.  We
calculate $(i^\chi)^*[\overline V]_{\C^*}$ using mixing spaces.

We will let 
$\chi$ be the character $z\mapsto z^a$ 
and use notation as in diagram~\eqref{E:mdiagram}:
\[
\xymatrix{
E^\chi_{\C^*} = \rho^* E\otimes \cO_{\Pbb^\infty}(-a)
\ar[r]^-{\pi^\chi} &
\Pbb^\infty\times X \ar[r]^-\rho & X\,.
}
\]
We denote mixing spaces by the subscript $\C^*$. By definition,
$[\overline V]_{\C^*}=[\overline V_{\C^*}]$ under the identification
$H^{ \C^*}_*(\Pbb(E^\chi \oplus \one))=H_*(\Pbb(E^\chi_{\C^*} \oplus \one)$.

The mixing space $\overline V_{\C^*}$ is a cone over $\Pbb(V_{\C^*})$.
By \Cref{lem:restr}, applied to the
open embedding $E^\chi_{\C^*} \to \Pbb(E^\chi_{\C^*} \oplus \one)$, and
\Cref{lem:eqsha},
\begin{align*}
(i^\chi)^*[\overline V_{\C^*}] &= \textrm{codimension-$k$ component of }
(\pi^\chi)^*( \Shadow_{E^\chi_{\C^*}  \oplus \one} ([\overline{V}_{\C^*})]) \\
&= \textrm{codimension-$k$ component of }
(\pi^\chi)^*( \Shadow_{E^\chi_{\C^*}}([\Pbb(V_{\C^*})])\,.
\end{align*}
There is a canonical isomorphism
\[ 
\Pbb(E_{\C^*}^\chi) = \Pbb(\rho^*(E)\otimes \cO_{\Pbb^\infty}(-a)) 
\cong \Pbb(\rho^*(E)) = \Pbb^\infty\times \Pbb(E) \/,
\] 
under which $\Pbb(V_{\C^*})\subseteq \Pbb(E_{\C^*}^\chi)$ is identified with
$\Pbb^\infty\times \Pbb(V)\subseteq \Pbb(\rho^*(E))$. We have
\[
\Shadow_{\rho^*(E)}([\Pbb^\infty\times\Pbb(V)]) =\sum_{j=0}^e
\rho^*(\aX^{k-j})\,,
\]
therefore, as $E^\chi_{\C^*} = \rho^* E\otimes \cO_{\Pbb^\infty}(-a)$,
\[
\Shadow_{E_{\C^*}^\chi}([\Pbb(V_{\C^*})]) 
=\sum_{j=0}^e c(\cO(-a))^j\cap \rho^*(\aX^{k-j})
\]
by \Cref{lem:lem42}~(iii). Extracting the codimension~$k$ component of this class,
we obtain
\[
(i^\chi)^*[\overline V_{\C^*}] = (\pi^\chi)^*\left(\sum_{j=0}^e c_1(\cO(-a))^j
\cap\rho^*(\aX^{k-j})\right)\in H_*(E^\chi_{\C^*})\,.
\]
Pulling back via the zero-section $\iota$, and viewing the result in the equivariant 
homology group, we deduce
\begin{equation}\label{eq:toe}
\iota^*([V]_{\C^*}) 
= \sum_{j=0}^e \chi^j \aX^{k-j} \in H_*^{\C^*}(X)\,,
\end{equation}
hence by~\eqref{eq:Shachi}
\[
\iota^*([V]_{\C^*})=(\Shadow_{E \oplus \one} ([\overline V]))^\chi 
\]
as needed.
\end{proof}

\subsection{Equivariant shadows}\label{sec:eqshadows}
If $X$ is endowed with the action of a torus $T$, and $\pi: E \to X$
is a $T$-equivariant vector bundle on $X$, then shadows of equivariant
classes in $H_*^T(\Pbb(E))$ may be defined in
$H_*^T(X)$. Explicitly, a $T$-stable cycle $C\subseteq E$ determines
a $T$-stable cycle $\overline C$ in the $T$-variety
$\Pbb(E \oplus \one)$, where $T$ acts trivially on $\one$. Using the
induced cycle $\overline C_T$ in $\Pbb(E_T\oplus \one)$, we let
\begin{equation}\label{eq:defeqsha}
\Shadow^T(C):= \Shadow_{(E_T\oplus \one)}([\overline C_T])\/;
\end{equation}
this class lives in the (ordinary) homology of the mixing space $X_T$,
and is therefore naturally an element of $H^T_*(X)$. Then
\Cref{lem:lem42}, \Cref{lem:eqsha}, \Cref{rem:zs}
and \Cref{lem:restr} extend
$T$-equivariantly, using equivariant classes throughout, as well as the key
fact that
\[ c_1^T(\cO_{E \oplus \one}(1)) \cap [\Pbb(E \oplus \one)]_T 
=  [\Pbb(E)]_T \quad \in H_*^T(\Pbb(E)) \/. \] This follows because 
$(E \oplus \one)_T = E_T \oplus \one$, since $T$ acts trivially on the 
line~$\one$.

For further use, we record the equivariant version of
Lemma~\ref{lem:lem42} (i), in the way we use it below. For a
$T$-stable cycle $C$ in $E$,
\begin{equation}\label{eq:equivlem42}
\Shadow^T(C) = c^T(E) \cap \Segre^T(C)\quad.
\end{equation}
Here, $\Segre^T(-)$ is an `equivariant Segre class' operator, defined as
follows: as above, the cycle~$C$ determines a $T$-stable cycle
$\overline C$ in $\Pbb(E \oplus \one)$, and
\begin{equation}\label{E:defsegre}
  \Segre^T(C):= \overline{q}_* (c^T(\cO_{E \oplus \one}(-1))^{-1} \cap [\overline C])
  =\sum_{i\geq 0} \overline{q}_* (c^T_1(\cO_{E \oplus \one}(1))^i  \cap [\overline C]) 
  \quad \/, 
\end{equation} 
where $\overline{q}:\Pbb(E \oplus \one) \to X$ is the projection.  In
the non-equivariant case, the Chern classes are nilpotent, therefore
the operator is well defined.  Equivariantly, this operator has values
in the completion $\hat{H}^T_*(X):=\prod_{i \leq \dim X} H^T_{2i}(X)$ of
$H^T_*(X)$.  The Segre operator will be key in \S\ref{s:pos},
where $X=G/B$ is a flag manifold. 
We refer to~\cite[Chapter~4]{fulton:IT} and~\cite{MR1393259}
for detailed information on Segre classes and operators in ordinary
Chow groups.

If in addition
$X$ is also endowed with a trivial $\C^*$ action, the definition given
in~\eqref{eq:homogdef} generalizes to give the homogenization of an 
{\em equivariant} (non-homogenous) class: for
  $\alpha = \sum_{j=0}^\ell \alpha^j \in H_*^T(X)$, with $\alpha^j$ of 
  codimension~$j$, and the choice of
  (a codimension) $k\geq \ell$, the homogenization
  $\alpha^\chi = \sum_{j=0}^\ell \chi^{k-j} \alpha^j$ 
is a class in
  $H^{T \times \C^*}_{2(\dim X - k)}(X)$.  
 If as above $E=E^\chi$ is given a $\C^*$-action by fiberwise
dilation with character $\chi$, then\/  
the natural projection $\pi: E \to X$, and the zero section $\iota: X \to E$
  are both $T \times \C^*$-equivariant, every $T$-stable
  cone $C\subseteq E$ is also $T \times \C^*$ stable, and since
  $\C^*$ acts trivially on~$X$,
  $H^{T\times \C^*}_*(X) \cong H^T_*(X)[\hbar]$. The analogue of
  \Cref{prop:shapb} is:

\begin{prop}\label{prop:eqpbsha}
  Let $\C^*$ act on fiber of $E$ by dilation with character $\chi$.
  Then for any $T$-stable cone $C\subseteq E$ of codimension $k$,
\[
  \Shadow^T(C)^\chi = \iota^*([C]_{T\times\C^*})\quad \/,
\]
where the left-hand side is the $\chi$ homogenization of degree $k$.
\end{prop}

\begin{proof}
Apply~\Cref{prop:shapb} to the mixing space~$X_T$.
\end{proof}


\section{Equivariant Chern-Schwartz-MacPherson
  classes}\label{sec:eqcsm}

\subsection{Preliminaries}\label{ss:bn}
Let $X$ be a scheme with a torus $T$ action. The group of
constructible functions $\cF(X)$ consists of functions
$\varphi = \sum_W c_W \one_W$, where the sum is over a finite set of
constructible subsets $W \subseteq X$ and $c_W \in \Z$ are integers.  A
group $\cF^T(X)$ of {\em equivariant\/} constructible functions (for
tori and for more general groups) is defined by Ohmoto in
\cite[\S2]{ohmoto:eqcsm}.  We recall the main properties that we need:
\begin{enumerate} 
\item If $W \subseteq X$ is a constructible set which is stable
  under the $T$-action, its characteristic function $\one_W$ is an
  element of $\cF^T(X)$. We will denote by $\cFTinv(X)$ the
  subgroup of $\cF^{T}(X)$ consisting of $T$-invariant constructible
  functions on $X$. (The group $\cF^T(X)$ also contains other
  elements, but this will be immaterial for us.)
\item Every proper $T$-equivariant morphism $f: Y \to X$ of algebraic
  varieties induces a homomorphism $f_*^T: \cF^T(Y) \to \cF^T(X)$ 
  defined by 
\begin{equation}\label{eq:pushf}
  f_*^T(\one_W)(x) = \chi(f^{-1}(x) \cap W)
\end{equation}
  for $x \in X$ and $W \subseteq Y$ a $T$-stable subvariety; here 
 $\chi$ denotes the topological Euler characteristic.~The
  restriction of $f_*^T$ to $\cFTinv(X)$ coincides with the
  ordinary push-forward $f_*$ of constructible functions; cf.~\cite[\S2.6]{ohmoto:eqcsm}.
\end{enumerate}

\begin{remark}\label{rem:incexc}
We recall that for complex algebraic varieties, the topological Euler characteristic agrees
with the Euler characteristic with compact support (see e.g., \cite[p.~95, p.~141]{MR1234037})
and it therefore satisfies additivity on disjoint unions of locally closed varieties. Every
constructible set $Z$ can be partitioned into a finite disjoint union of locally closed
subvarieties $V_i$ and one can define $\chi(Z):=\sum_i \chi(V_i)$; this is well-defined since
any two partitions have a common refinement.
If $Z_1$ and $Z_2$ are constructible sets,
then there is a partition of $Z_1\cup Z_2$ into locally closed subvarieties extending
partitions of $Z_1$, $Z_2$, and $Z_1\cap Z_2$, and it follows that
\[
\chi(Z_1\cup Z_2)= \chi(Z_1) + \chi(Z_2) - \chi(Z_1\cap Z_2)\,,
\]
that is, the Euler characteristic of constructible functions satisfies inclusion-exclusion.

In other words, the Euler characteristic defines a group homomorphism $\chi(X;-)$ from 
the group of constructible functions of a variety $X$ to $\Z$. For instance, 
\eqref{eq:pushf}~holds for any $T$-stable constructible subset $W$ of $Y$.

This group homomorphism  $\chi(X;-): \cF(X)\to \Z=\cF(\pt)$ can also be seen as the 
special case for $Y=\pt$  of the push-forward $f_!=f_*: \cF(X)\to \cF(Y)$ for a morphism 
$f: X\to Y$ of complex algebraic varieties, induced  from the corresponding 
transformations $f_!,f_*$ of the Grothendieck groups $K_0(D^b_c(-))$ of (algebraically) 
constructible sheaf complexes (of vector spaces)  coming from the derived functors 
$Rf_!,Rf_*$. Here one uses the group epimorphism
\[
\chi_{stalk}: K_0(D^b_c(-))\to \cF(-)
\]
given by the stalkwise Euler characteristic as 
in~\cite[Section~2.3, Section~6.0.6]{schurmann:book}.
Then the equality $f_!=f_*: K_0(D^b_c(X))\to K_0(D^b_c(Y))$ 
from~\cite[Equation~(6.41), p.~413]{schurmann:book} for a morphism $f: X\to Y$ implies 
the corresponding equality $f_!=f_*: \cF(X)\to \cF(Y)$ for constructible functions 
(see~\cite[Equation~(6.42), p.~413]{schurmann:book}).
\qede\end{remark}

Ohmoto proves \cite[Theorem~1.1]{ohmoto:eqcsm} that there is an
equivariant version of MacPherson's transformation
$\csT: \cF^T(X) \to H_{2*}^T(X)$ (the image is in even homology degrees) 
that satisfies
$\csT( \one_X) = c^T(T_X) \cap [X]_T$ if $X$ is a non-singular
variety, and that is functorial with respect to proper push-forwards.
The last statement means that for all proper $T$-equivariant morphisms
$f:Y\to X$ the following diagram commutes:
\[ 
\xymatrix{ 
\cF^T(Y) \ar[r]^{\csT} \ar[d]_{f_*^T} & H_*^T(Y) \ar[d]^{f_*^T} \\ 
\cF^T(X) \ar[r]^{\csT} & H_*^T(X) 
}
\]
While in our main application $X$ will be a
  projective (flag) variety, we note that $\csT$ may be defined for
  quasi-projective schemes, or, more generally, for separated
  algebraic spaces \cite{ohmoto:eqcsm,ohmoto:note}.
\begin{defin}\label{def:CSM}
  Let $Z$ be a $T$-stable constructible subset of $X$. We denote by
  $\csmT(Z):=\csT(\one_{Z}) \in H_*^T(X)$ the {\em equivariant
    Chern-Schwartz-MacPherson (CSM) class,\/} and for $Z$ a
  $T$-stable algebraic subvariety
  of $X$ by $\cMaT(Z):=\csT(\Eu_Z) {~\in H_*^T(X)}$ the {\em
    equivariant Chern-Mather\/} class of $Z$.
\qede\end{defin}

Here, $\Eu_Z$ is MacPherson's local Euler obstruction.
This is a constructible function, equal to $1$ at non-singular 
points of $Z$. It may be defined using transcendental methods
in analytic topology (\cite[\S3]{macpherson:chern}) as well as 
algebraically~(\cite[Example~4.2.9]{fulton:IT}). Both CSM and
Chern-Mather classes depend on the
  chosen ambient space $X$. However, if $\overline{Z}$ is the closure
  of $Z$ in $X$, then the inclusion $\overline{Z} \subseteq X$ is
  proper, and one may view these classes as (non-homogenous) elements
  of $H_*^T(\overline{Z})$; the corresponding classes in any
  $T$-stable subvariety $W$ of $X$ containing $Z$ may be obtained
  by pushing-forward these classes (by the functoriality of
  $\csT$). We will often omit the dependence of the ambient space,
  when this space is clear from the context.  
  Both $\csmT(Z)$ and
  $\cMaT(\overline{Z})$
equal $[\overline Z]_T +$ lower dimensional terms, and both classes
agree with $c^T(TZ)\cap [Z]_T$
 if $Z\subseteq X$ is a nonsingular  subvariety (which is closed by our conventions).
In~\cite[\S4.3]{ohmoto:eqcsm} Ohmoto gives an explicit
geometric construction of the equivariant Chern-Mather class.

It will be useful to consider `signed' versions $\chc_*$, $\chc_*^T$ of 
MacPherson's and Ohmoto's natural transformations.
These are defined by changing the sign of components of odd (complex)
dimension. Thus for every constructible function $\varphi$ on $X$ we set
\[
\chc_*(\varphi) = \sum_{k\ge 0} \chc_*(\varphi)_k:= \sum_{k\ge 0}(-1)^k c_*(\varphi)_k\/,
\]
where $c_*(\varphi)_k\in H_{2k}(X)$ is the component of complex dimension~$k$, 
and similarly for all invariant $\varphi$ we define
\begin{equation}\label{E:checkc}
\chc_*^T(\varphi) = \sum_{k\ge 0} \chc^T_*(\varphi)_k:=
\sum_{k\ge 0} (-1)^k c^T_*(\varphi)_k\/.
\end{equation}

Note that $c^T_*(\varphi)_k=0$ for $k<0$ 
  by~\cite[\S4.1]{ohmoto:eqcsm}.
Also,
\[
  \int_X c^T(\varphi)_0 = \chi(X;\varphi) \quad \text{for $X$ compact 
  and $\varphi\in \cFTinv(X)$}
\]
by functoriality of $\csT$ and $\Z$-linearity ~\cite[p.~127]{ohmoto:eqcsm}.
Here $\chi(X;\varphi)$ is the Euler characteristic weighted by $\varphi$,
induced from the Euler characteristic of constructible sets, cf.~Remark~\ref{rem:incexc}.

In the non-equivariant setting, this signed Chern class transformation
appears implicitly in e.g., work of Sabbah \cite{sabbah:quelques} and
Sch{\"u}rmann \cite{schurmann:lecture} (see also
\cite{MR1063344,PP:hypersurface,schurmann:lecture,
  schurmann:transversality}) where MacPherson's transformation is
constructed via Lagrangian cycles in the cotangent bundle of $X$. The
equivariant version of this construction is discussed below
in~\S\ref{ss:scsm}.
`Dual' CSM and Chern-Mather classes are defined by setting,
for a subvariety $Z\subseteq X$,

\begin{equation}\label{E:defdual}
\csmTv(Z):= (-1)^{\dim Z} \chc_*^T(\one_Z)\quad,\quad
\cMaTv(Z):= (-1)^{\dim Z} \chc_*^T(\Eu_Z)\quad.
\end{equation}
The sign is introduced so that
if $Z$ is $T$-stable, irreducible, and nonsingular, then both
classes agree with the equivariant Chern class of the cotangent bundle
of $Z$, $c^T(T^*(Z))\cap [Z]_T$.
For complete flag manifolds, a geometric interpretation of dual
  CSM classes in terms of Poincar{\'e} duality will be given 
  in~\S\ref{s:homogenized}.

\subsection{CSM classes, shadows of characteristic cycles, and an
  intersection formula}\label{ss:scsm}
In this section we recall a construction of MacPherson's natural
transformation by means of {characteristic cycles,\/} and extend this
construction to the equivariant setting.  In the non-equivariant case
this construction appears in (among others)
\cite{sabbah:quelques,MR1063344,PP:hypersurface,aluffi:shadows,schurmann:transversality}.
{Our main result is \Cref{prop:shadows}, which realizes this
  construction in terms of equivariant shadows. This also gives a {\em
    Lagrangian approach} to Ohmoto's construction of the equivariant
  MacPherson's transformation \cite{ohmoto:eqcsm}.}

In addition, we record in \Cref{thm:intersection} certain
  equivariant intersection formulae for MacPherson's transformation,
  generalizing the non-equivariant statements proved in
  \cite{schurmann:transversality}.  These will be needed in the
  proof of `geometric orthogonality' for flag manifolds in~\S\ref{ss:geomorth}.
The details required to extend the proofs from \cite{schurmann:transversality}
to the equivariant setting are given in the Appendix.

In this section $X$ will denote a smooth (complex) algebraic variety
with an action of a torus $T$.  As before, we state our results in
equivariant Borel-Moore homology; {\em mutatis mutandis,\/} they hold
in the Chow group.
The construction is illustrated in the following diagram
(cf.~\cite[\S3]{schurmann:transversality}); the notation
is explained next.

\begin{equation}\label{eq:conormal}
\vcenter{\xymatrix@C=50pt@R=30pt{
\cFTinv(X) \ar@{=}[d] & Z_*^T(X) \ar[l]^\sim_{\cEu}
\ar[d]^\wr_{\cn} \ar[r]^\cMaTv & H_{2*}^T(X) \ar@{=}[d] \\
\cFTinv(X) \ar[r]_\sim^{\CC} & L_T(X) \ar[r]^{\Shadow^T } & H_{2*}^T(X)
}}
\end{equation}

Here $Z_*^T(X)$ denotes the group of $T$-stable cycles in $X$,
while $L_{T}(X)$ denotes the additive group of $T$-stable conic
Lagrangian cycles in the cotangent bundle $T^*(X)$ of $X$. (This is a
$T$-equivariant bundle, where the $T$-action is induced from the
$T$-action on~$X$.) The elements of $L_T(X)$ are $\Z$-linear
combinations of {\em conormal cycles}
  $T^*_Z X:= \overline{T^*_{Z^\text{reg}} X}\subseteq T^*(X)$, where 
  $Z \subseteq X$ is a $T$-stable subvariety and $Z^\text{reg}$ is the smooth
  part of $Z$.

The top maps are the `signed' local Euler obstruction, defined on
subvarieties $Z$ by $\cEu_Z := (-1)^{\dim Z} \Eu_Z$ and
extended to cycles by linearity, and the
dual equivariant Chern-Mather class $\cMaTv$, defined as
in~\eqref{E:defdual} on $T$-stable varieties and extended by
linearity to $T$-stable cycles.  The homomorphism $\cEu$ is an
isomorphism, and the composition
\[ 
\cMaTv \circ \cEu^{-1} = \chc_*^T\/ 
\] 
is the signed equivariant MacPherson transformation.
(Cf.~\cite[Proposition~4.3]{ohmoto:eqcsm}. Ohmoto works with
non-signed classes, but the signed versions are convenient for us as
they come up naturally in the context of characteristic cycles in the
cotangent bundle $T^*(X)$.)  The map $\cn: Z_*^T(X) \to L_T(X)$ takes an
irreducible cycle $Z$ to its conormal cycle $T^*_Z X$.
This map is a group isomorphism; see e.g., \cite[Lemma~3]{MR1063344}
or \cite[Theorem~E.3.6]{HTT}.  By composition we obtain an induced
`characteristic cycle' map $\CC: \cFTinv(X) \to L_{T}(X)$
determined on irreducible $T$-stable cycles $Z$ by
\begin{equation}\label{eq:EoCc}
\CC(\cEu_Z) =  T^*_Z X
\end{equation}
(see~\cite[(11), page~67]{PP:hypersurface}).  Since both maps
$\cn$,  $\cEu$ are isomorphisms, the characteristic cycle map $\CC$
is a group isomorphism as well.  For a constructible function
$\varphi$, the image $\CC(\varphi)$ is a conic Lagrangian cycle in
$T^*(X)$ called the {\em characteristic cycle\/} of $\varphi$; this
cycle is clearly $T$-stable if $\varphi\in \cFTinv(X)$.

The map $\Shadow^T: L_T(X) \to H_{2*}^T(X)$ in the diagram is the
`equivariant shadow' operation defined in~\S\ref{sec:eqshadows}, 
see~\eqref{eq:defeqsha}.

\begin{theorem}\label{prop:shadows}
  Diagram~\eqref{eq:conormal} commutes, i.e.,
\[
\Shadow^T(\CC(\varphi)) = \chc^T_*(\varphi)
\]
for every invariant constructible function~$\varphi$.
\end{theorem}

By linearity and the construction of $\CC$, this statement is equivalent to the assertion
that
\[
\cMaTv(Z)=\Shadow^T(T^*_ZX)
\]
for all $T$-stable subvarieties $Z$ of $X$. In particular, for $Z=X$ we recover by (the 
equivariant version of)~\Cref{rem:zs} the normalization
\[
\cMaTv(X)=\Shadow^T(T^*_XX)=c^T(T^*(X))\cap [X]_T\:.
\]

\Cref{prop:shadows} also shows that the map $\Shadow^T$ coincides with
a map defined by Ginzburg in \cite[\S{A.3}]{ginzburg:characteristic},
using $\C^*$-equivariant K-theory on $T^*(X)$.
In this respect, \Cref{prop:shadows} can be regarded as an
alternative to Ginzburg's construction; see also \Cref{prop:eqpbsha}.

The non-equivariant version of \Cref{prop:shadows} is
\cite[Lemma~4.3]{aluffi:shadows}; this is essentially a reformulation
of \cite[(12), page~67]{PP:hypersurface}, which in turn is based
on a calculation of Sabbah~\cite{sabbah:quelques}.  For another
approach to this non-equivariant version of \Cref{prop:shadows} see
also \cite{schurmann:lecture}.  In the present context, the connection
between shadows and these formulae is given by
  \Cref{eq:equivlem42} in~\S\ref{sec:eqshadows}, applied to the case when
  $C \in L_T(X)$.

\begin{remark} As mentioned in the introduction, a reader who is not familiar 
with Ohmoto's paper~\cite{ohmoto:eqcsm} (or with Ginzburg's 
\cite[Appendix]{ginzburg:characteristic} in the non-equivariant case)
can take~\Cref{prop:shadows}, or its reformulations~\Cref{cor:segree}  
 and~\Cref{thm:zeropb}, as the starting point for a definition of the
equivariant Chern class transformation $\csT$.
In future work we will explore further
functoriality properties of characteristic cycles in the equivariant context, 
aiming to prove directly the functoriality properties for the formulae 
in~\Cref{cor:segree} and~\Cref{thm:zeropb} (as already done in the 
Appendix~\S\ref{s:app} for a non-characteristic pullback result).
\end{remark}

\begin{proof}[Proof of \Cref{prop:shadows}]
  Let $U$ denote an approximation space to $ET$
  (see \cite{edidin.graham:eqchow}, \cite{ohmoto:eqcsm}), and denote
  by $u$ the projection $U\times_T X\to U/T$; 
 this is a locally trivial fibration,
  with fibers isomorphic to $X$. We have the 
following  cartesian diagram of $T$-equivariant maps,
with $pr$ the projection and $q,q'$ the quotient maps:
\[
\xymatrix{
U\times_T X \ar[r]^-u & U/T  \\
U\times X \ar[r]_-{pr} \ar[u]_-{q'} & \:\:U\:.\ar[u]^-q
}
\]
  The relative
  cotangent bundle  $T^*u=U\times_T T^*(X)$ is the $U$-approximation of the bundle
  $T^*(X)_T$. Every invariant constructible function $\varphi$ on~$X$
  determines a constructible function on $U\times_T X$, agreeing with
  $\varphi$ on the fibers of $u$; we denote this function by~$\varphi_U$.
It is uniquely characterized by the requirement that
\[
q'^*(\varphi_U)=pr'^*(\varphi) \in \cFTinv(U\times X)\:,
\]
with $pr': U\times X\to X$ the other projection.
  Ohmoto defines~$c^T_*(\varphi)$ as follows:
\begin{equation}\label{eq:Ohmdef}
c^T_*(\varphi):= \varinjlim_U\, c(T(U) \times_T X)^{-1}\cap c_*(\varphi_U)\,,
\end{equation}
(\cite[p.~122 and Definition~3.2]{ohmoto:eqcsm}).
The (dual of the) exact sequence of differentials for $q$
pulls back to an exact sequence of equivariant vector bundles on $U\times X$:
\[
\xymatrix{
0 \ar[r] & (Tq)\times X \ar[r] & T(U)\times X \ar[r] & q^*T(U/T)\times X \ar[r] & 0\,
}
\]
inducing the exact sequence
\[
\xymatrix{
0 \ar[r] & (Tq)\times_T X \ar[r] & T(U) \times_T X \ar[r] & u^*(T(U/T)) \ar[r] & 0
}
\]
on $U\times_T X$. 
Now, $Tq\cong U\times \mft$ with the adjoint action of $T$ on its Lie algebra $\mft$
(see~\cite[Lemma~A.1, p.~580]{EG:nonabelian} or~\cite[pp.~2218-2219]{MS:toric}).
In our case the adjoint action of the torus $T$ on $\mft$  is trivial, since $T$ is abelian, 
so that the bundle $(Tq)\times_T X$ is  trivial. Therefore, the bundle
$T(U) \times_T X$ is an extension of~$u^* (T(U/T))$ by a trivial bundle, and we can 
rewrite Ohmoto's definition~\eqref{eq:Ohmdef} as
\begin{equation}\label{eq:Ohmdefcorr}
c^T_*(\varphi)= \varinjlim_U u^* c(T(U/T))^{-1}\cap c_*(\varphi_U)\,.
\end{equation}
Elements of $H^T_{2k}(X)$ are limits of classes in complex codimension $\dim X-k$ in
$U\times_T X$, i.e., of dimension $k+\dim (U/T)$, therefore~\eqref{eq:Ohmdefcorr} gives
\[
\chc^T_*(\varphi):= \varinjlim_U u^* c(T^*(U/T))^{-1}\cap (-1)^{\dim U/T}\chc_*(\varphi_U)\,.
\]
By~\cite[Lemma~4.3]{aluffi:shadows},
\[
\chc_*(\varphi_U)=\Shadow_{T^*(U\times_T X)\oplus\one}(\overline{\CC(\varphi_U)})
\]
in $H_*(U\times_T X)$. 
 In order to complete the proof we have to show that
\begin{multline*}
(-1)^{\dim U/T}\varinjlim_U u^* c(T^*(U/T))^{-1}
\Shadow_{T^*(U\times_T X)\oplus\one}(\overline{\CC(\varphi_U)}) \\
=\Shadow^T(\CC(\varphi))
\end{multline*}
for all $T$-invariant constructible functions $\varphi$. 
By linearity, we may assume $\varphi=\cEu_Z$ for a $T$-invariant subvariety $Z$ of $X$,
so that $(-1)^{\dim U/T}\varphi_U= \cEu_{U\times_T Z}$. Then $\CC(\varphi)
=T^*_ZX$ and $(-1)^{\dim U/T}\,\CC(\varphi_U)=T^*_{U\times_T Z}(U\times_T X)$. 
By definition (see \eqref{eq:defeqsha}),
\begin{equation}\label{eq:Shaphi}
\Shadow^T(T^*_ZX)= \Shadow_{(T^*X)_T\oplus \one}(\overline{(T^*_ZX)_T})\/,
\end{equation}
where $\overline{(T^*_ZX)_T}$ is the closure in $\Pbb(T^*(X)_T\oplus \one)$
of the cycle $(T^*_ZX)_T$ determined by $T^*_ZX$ in the bundle $T^*(X)_T$
over the mixing space. 
As recalled above, the $U$-approximation of this bundle is the relative cotangent bundle 
$T^*u=U\times_T T^*(X)$, and the $U$-approximation of $(T^*_ZX)_T$ is
$U\times_T T^*_ZX$. We rewrite \eqref{eq:Shaphi} as
\[
 \Shadow^T(T^*_ZX) =\, \varinjlim_U\, \Shadow_{T^*u\oplus \one}(\overline{U\times_T T^*_ZX})\/,
\]
and what is left to prove is the following identity of classes
in $H_*(U\times_T X)$:
\begin{multline}\label{eq:cl27}
\Shadow_{T^*(U\times_T X)\oplus\one}(\overline{T^*_{U\times_T Z}(U\times_T X})) \\
=u^* c(T^*(U/T))\cap \Shadow_{T^*u\oplus\one}(\overline{U\times_T T^*_ZX})\/.
\end{multline}
If $Z=X$, then the conormal
cycle $T^*_XX$
is the zero-section of $T^*(X)$, and the cycles $ T^*_{U\times_T X}(U\times_T X)$,
$ U\times_T T^*_XX$ are the zero-sections of $T^*(U\times_T X)$, $T^*u$,
respectively. In this case, \eqref{eq:cl27} amounts to
\[
c(T^*(U\times_T X))\cap [U\times_T X] = u^* c(T^*(U/T))\, c(T^*u) \cap [U\times_T X]
\]
(\Cref{rem:zs}), which follows from the Whitney formula 
 and the exact sequence of differentials for $u$:
\begin{equation}\label{exseq-relative-u}
\xymatrix{
0 \ar[r] & u^* T^*(U/T) \ar[r] & T^*(U\times_T X) \ar[r] & T^*u \ar[r] & 0\,.
}
\end{equation}

Therefore, we may assume that $Z\neq X$.
We claim that 
 $T^*_{U\times_T Z}(U\times_T X)$ only meets $u^* T^*(U/T)$ in the
  zero-section. 

Indeed, the smooth map $q':U\times X\to U\times_T X$ induces a closed embedding
\[
\xymatrix{
t\colon\, q'^* T^*(U\times_T X) \ar@{^(->}[r] & T^*(U\times X) = T^*(U)\oplus T^*(X)
}
\]
of bundles over $U\times X$. Via this embedding, $q'^*u^* T^*(U/T)$ is mapped to 
 a subbundle of
$T^*(U)\oplus 0$, while 
\[ t(q'^*(T^*_{U\times_T Z}(U\times_T X))=0\oplus T^*_ZX \/, \]
as may be verified by chasing the pull-back via $q'$ of the exact sequence
\eqref{exseq-relative-u}.
(This equality is also a special case of \Cref{thm:non-cc}, 
utilizing that $q'$ is non-characteristic with respect to any closed cone in 
$T^*(U\times_T X)$, because $q'$  is smooth.)
The claim follows.
From $\eqref{exseq-relative-u}$ it also follows that $U\times_T T^*_ZX$ is the image
of $T^*_{U\times_T Z}(U\times_T X)$ in~$T^*u$.

Since the intersection of $T^*_{U\times_T Z}(U\times_T X)$ and $u^* T^*(U/T)$ is contained
in the zero-section, the projectivization $\Pbb(T^*_{U\times_T Z}(U\times_T X))$ is disjoint
from $\Pbb(u^* T^*(U/T))$. 
By~\Cref{lem:lem42}~{(iv)}, we have
\begin{multline*}
\Shadow_{T^*(U\times_T X)}\left(\Pbb(T^*_{U\times_T Z}(U\times_T X))\right) \\
=u^* c(T^*(U/T))\cap \Shadow_{T^*u}\left(\Pbb(U\times_T T^*_ZX)\right)\/.
\end{multline*}
Applying~\Cref{lem:eqsha} to both sides of this identity gives~\eqref{eq:cl27} as needed.
\end{proof}

\Cref{prop:shadows} and identity~\eqref{eq:equivlem42} imply the following
key result, extending analogous results 
from~\cite{sabbah:quelques,PP:hypersurface}
to the equivariant setting.

\begin{corol}\label{cor:segree} 
Let $\varphi \in \cFTinv(X)$ be an invariant constructible
function. Then the following identity holds in $H_*^T(X)$:
\[ 
\chcT( \varphi) = c^T(T^*(X)) \cap \Segre^T(\CC(\varphi)) \/. 
\] 
\end{corol}

The {\em signed Segre-MacPherson class\/} of a $T$-invariant constructible
function $\varphi \in \cFTinv(X)$ is the class
\begin{equation}\label{E:signed-segre} 
  \chssmT(\varphi):= \frac{\chc_*^T(\varphi)}{c^T(T^*(X))} 
  = \Segre^T(\CC(\varphi)) \in \hat{H}^T_*(X)\/.
\end{equation} 
The (unsigned) classes
\begin{equation}\label{E:segre-new} 
\ssmT(\varphi):= \frac{\csT(\varphi)}{c^T(TX)}  \in \hat{H}^T_*(X)\:,
\end{equation} 
called {\em Segre-MacPherson} (SM) classes (see \cite[\S5.3]{ohmoto:eqcsm}), 
are related to the study of Thom polynomials. 
In the non-equivariant case they  were studied in~\cite{aluffi:inclusionI} and are
denoted here by $\chssm(\varphi)$ resp.,  $\ssm(\varphi)$.
For a ($T$-stable) constructible subset $Z$ of $X$, we denote by $\ssm(Z):=\ssm(\one_Z)$
(resp., $\ssmT(Z):=\ssmT(\one_Z)$) the {\em (equivariant) Segre-MacPherson class 
of $Z$.}

Next we will state an intersection formula for the equivariant MacPherson's 
transformation, which we will use in~\S\ref{s:homogenized} to compute 
Poincar{\'e} duals of equivariant CSM classes of Schubert cells
  (see~\Cref{thm:gporth}).  In the non-equivariant case, this formula was 
  proved in~\cite{schurmann:transversality}.  We indicate how to extend 
  the arguments to the equivariant setting in the Appendix.  In fact, there
  we will use \Cref{cor:segree} to extend to the equivariant case the
  more general
  `non-characteristic pullback' results for (signed) Segre-MacPherson 
  classes.

We remind the reader that in this section $X$ is assumed to be a
smooth complex algebraic variety endowed with an action of a torus $T$.
In particular, an intersection product is defined in $H^T_*(X)$.
Also, we say that two locally closed nonsingular subvarieties $S$, $S'$ of $X$ are {\em transversal\/} if
$T_x(S)+T_x(S')=T_x(X)$ for all $x\in S\cap S'$.

\begin{theorem}\label{thm:intersection}
  Let $\alpha\in \cFTinv(X)$, resp., $\beta\in \cFTinv(X)$ be
  constructible with respect to algebraic Whitney stratifications
  $\cS:= \{ S \subseteq X \}$, resp.,
  $\cS':= \{ S' \subseteq X \}$ of $X$, 
i.e., $\alpha|S$ and $\beta|S'$ are constant for all strata $S\in \cS$ and 
$S'\in \cS'$. Assume
that each stratum
  $S\in \cS$ is transversal to each stratum
  $S'\in \cS'$.
Then
  \begin{equation*}
  \csT(\alpha\cdot \beta)= \csT(\alpha) \cdot
  \ssmT(\beta) \in H^T_*(X)\subseteq \hat{H}^T_*(X)\:.
  \end{equation*}
  In particular, if $X$ is compact, then
  \[
    \langle \csT(\alpha), \ssmT(\beta) \rangle :=\int_X
    \csT(\alpha) \cdot \ssmT(\beta) =\chi(X;\alpha\cdot
    \beta) \/.
  \]
\end{theorem}

Note that if $X$ is compact with a finite $T$-fixed point set $X^T$, then
\begin{equation*}
  \chi(X;\alpha\cdot \beta) = \chi(X^T;\alpha\cdot \beta) 
  = \sum_{x\in X^T} \alpha(x)\cdot \beta(x) \/; 
\end{equation*}
see~\cite[Corollary~3.2.2, p.~174]{schurmann:book}.  This is also
equivalent to the fact that $\chi(Z)$~is given by the number $|Z\cap X^T|$ 
of $T$-fixed points in $Z$ for any locally closed $T$-stable
subvariety $Z\subseteq X$.

\begin{remark}
  We will apply \Cref{thm:intersection} in the situation when $X=G/P$
  is a partial flag manifold and the $T$-stable stratification
  $\cS$ (resp., $\cS'$) is given by the (opposite)
  Schubert cells (see~\cite[Theorem~1.4]{richardson:double}). Since these
  finitely many Schubert cells are also orbits for a corresponding
  Borel subgroup in~$G$, the stratifications are automatically Whitney
  regular; indeed, Whitney regularity holds generically, and then
  by equivariance also everywhere along a Borel orbit.
 A constructible function is invariant with respect to such a Borel 
subgroup if and only if it constructible with respect to the corresponding Whitney 
stratification by these Borel orbits.
\qede\end{remark}

\section{Pulling back characteristic cycles by the zero section}\label{s:cc}

\subsection{Homogenized CSM classes are pull backs of characteristic
  cycles}
Let $X$ be a nonsingular variety endowed with a $T$-action.
\Cref{prop:eqpbsha}, applied to the cotangent bundle $T^*(X)$, gives
another construction of the map $L_T(X) \to H_*^T(X)$
in diagram~\eqref{eq:conormal}, using the
equivariant pull-back via the zero section $\iota:X \to T^*(X)$ of the
cotangent bundle. Recall that we view $T^*(X)$ as a
$T \times \C^*$-equivariant bundle, where the $T$-action is induced
from the $T$-action on $X$ as in~\S\ref{ss:scsm} and the $\C^*$
factor acts on the fibers of $T^*(X)$ by dilation with
character~$\chi$. In this section we focus on the case when
$\chi$ is the character $\hbar^{-1}$ given by $z \mapsto z^{-1}$.
\Cref{prop:eqpbsha} implies the following statement.

\begin{corol}
  Let $\C^*$ act on fiber of $T^*(X)$ by dilation with character
  $\hbar^{-1}$.  Then for all $C\in L_T(X)$,
\[
\Shadow^T(C) = \iota^*([C]_{T\times\C^*}) |_{\hbar \mapsto -1}\quad.
\]
\end{corol}

By the commutativity of diagram~\eqref{eq:conormal}, the same result
implies a direct realization of the homogenized CSM class of a
constructible function in terms of the equivariant pull-back:

\begin{theorem}\label{thm:zeropb} 
  Let $\iota: X \to T^*(X)$ be the zero section, and let $\C^*$ act on
  the fibers of $T^*(X)$ by the character $\hbar^{-1}$. Then the
    following holds for any $\varphi\in \cFTinv(X)$:
\[
\iota^*[\CC(\varphi)]_{T \times \C^*} = \csT(\varphi)^\hbar
\in H_0^{T \times \C^*} (X) \/, 
\]
 where $\csT(\varphi)^\hbar$ denotes the homogenization of
degree $\dim X$ (cf.~\eqref{eq:homogdef}).
\end{theorem}
\begin{proof}
By~\Cref{prop:eqpbsha} and~\Cref{prop:shadows},
\[
\iota^*[\CC(\varphi)]_{T \times \C^*} = \Shadow^T(\CC(\varphi))^{-\hbar}
=\chc^T_*(\varphi)^{-\hbar}\quad.
\]
By definition of homogenization and of the signed Chern class,
\[
\chc^T_*(\varphi)^{-\hbar}
=\sum_{j=0}^{\dim X} (-\hbar)^j  \left((-1)^j \csT(\varphi)_j\right)
=\sum_{j=0}^{\dim X} \hbar^j  \csT(\varphi)_j = \csT(\varphi)^{\hbar} \:.
\]
(Note that here the class is indexed by dimension, while in the 
definition given in~\eqref{eq:homogdef} 
it is indexed by codimension.)
\end{proof}

\begin{example} 
  Let $X= \Pbb^1$ and consider the constructible function
  $\one_{\Pbb^1}$. For simplicity, we will only work
  $\C^*$-equivariantly. Then
\[ 
{c}_*(\one_{\Pbb^1}) =
[\Pbb^1] + 2 [\pt] = c(T \Pbb^1) \cap [\Pbb^1] \/.
\]
By definition of homogenization, 
\[
  c_*(\one_{\Pbb^1})^{\hbar} = \hbar [\Pbb^1] + 2 [\pt] \/. 
\] 
On the other hand, by the self-intersection formula,
\[
\iota^* (\iota_* [\Pbb^1]_{\C^*}) 
= c_1^{\C^*}(T^* (\Pbb^1)) \cap [\Pbb^1]_{\C^*} 
= (c_1(T^* (\Pbb^1)) - \hbar) \cap [\Pbb^1] 
= - \hbar [ \Pbb^1] - 2[\pt] \/. 
\] 
Together with the fact that $\CC(\one_{\Pbb^1}) = -[T^*_{\Pbb^1}\Pbb^1] 
= - \iota_*[\Pbb^1]_{\C^*}$, this implies that 
\[ 
\iota^* [\CC(\one_{\Pbb^1})]_{\C^*} =
c_* (\one_{\Pbb^1})^{\hbar} 
\] 
as claimed.
\qede\end{example}

Specializing~\Cref{thm:zeropb} to the constructible functions
$\varphi=\one_Z$ and $\varphi=\Eu_Z$ gives the following.

\begin{corol}\label{cor:CSMpb} 
Let $Z\subseteq X$ be a $T$-stable constructible subset,
and let $\C^*$~act on the fibers of $T^*(X)$ by the character $\hbar^{-1}$. 
Then the homogenized CSM class satisfies 
\[ 
\csmT(Z)^{\hbar} =
\iota^* [\CC(\one_{Z})]_{T \times \C^*}~\in H_0^{T \times \C^*}(X) \/.
\] 
If $Z\subseteq X$ is a $T$-stable subvariety then
the homogenized Chern-Mather class satisfies 
\[
  \cMaT({Z})^{\hbar} = (-1)^{\dim {Z}}\iota^* [T^*_{Z} X]_{T \times
    \C^*} ~\in H_0^{T \times \C^*}(X) \/.
\] 
\end{corol}

\begin{remark}
    If one further specializes \Cref{thm:zeropb} to the characteristic
    function $\varphi=\one_X$ and 
 forgets the $T$-action,
then one obtains the classical index formula 
 for a nonsingular compact variety $X$:
\[
(-1)^{\dim X} \int_X \iota^* [T^*_X X] = \int_X c(\one_X)_0= \chi(X) \/, 
\]
where $\chi(X)$ is the Euler characteristic.
(Note that as $X$ is nonsingular, $\CC(\one_X)=\CC(\Eu_X)= (-1)^{\dim X} T^*_XX$,
cf.~\eqref{eq:EoCc}.)
\qede\end{remark}

\subsection{Characteristic classes of Lagrangian cycles are pull-backs}\label{ss:Ginz}
Next, we recall a commutative diagram considered  also  by Ginzburg in
\cite[Appendix]{ginzburg:characteristic}, which is largely based on
results from~\cite{beilinson.bernstein:localisation,
brylinski.kashiwara:KL,kashiwara.tanisaki:characteristic}. 
In the specific case of flag manifolds, this will be used in~\S\ref{CSMVerma} below.

\begin{equation}\label{E:diagram}
\vcenter{\xymatrix@C=40pt{
\Perv(X) \ar[d]_{\chi_\text{stalk}} & {\Mod_\text{rh}}(\caD_X) \ar[l]_{\DR}^\sim \ar[d]^\Char \\
\cF (X) \ar[r]^\CC_\sim & L(X) \ar[r]^{ c_*^{\vee, \mathrm{Gi}}} & H_*(X) \/.
}}
\end{equation}
Here $\Mod_\text{rh}(\caD_X)$, resp., $\Perv(X)$ denote the (sets of objects
of the) Abelian categories of algebraic holonomic $\caD_X$-modules with 
regular singularities, resp., perverse (algebraically) 
constructible complexes of sheaves of $\C$-vector spaces on $X$. The functor
$\DR$ is defined on $M^\cdot\in \Mod_\text{rh}(\caD_X)$ by
\[
\DR(M^\cdot) = R\caH om_{\caD_X}(\cO_X,M^\cdot)[\dim X]\/,
\]
that is, it computes the DeRham complex of a holonomic module (up to a
shift), viewed as an \emph{analytic} $\caD_X$-module. This functor
realizes the Riemann-Hilbert correspondence, and is an equivalence. We
refer to e.g., \cite{kashiwara.tanisaki:characteristic,
  ginzburg:characteristic} for details. The left map $\chi_\text{stalk}$
computes the stalkwise Euler characteristic of a constructible
complex, and the right map $\Char$ gives the characteristic cycle of a
holonomic $\caD_X$-module. The map $\CC$ is the characteristic cycle
map for constructible functions from diagram~\eqref{eq:conormal}.  
The commutativity of diagram \eqref{E:diagram} is shown in
\cite{ginzburg:characteristic} using deep $\caD$-module
techniques; it also follows from \cite[Ex.~5.3.4 on
pp.~359-360]{schurmann:book}.  Also note that DR in \eqref{E:diagram} 
factors through the corresponding Grothendieck group, so it may
also be applied to complexes of regular holonomic $\caD$-modules.

 We define the map $c_*^{\vee, \mathrm{Gi}}$ by
\begin{equation}\label{eq:Ginim1}
c_*^{\vee, \mathrm{Gi}} := \iota^*_{\hbar =-1}\/,
\end{equation}
i.e., by the  specialization at
$\hbar =-1$ of the pull-back via the zero-section map $\iota: X \to T^*(X)$,
with $\C^*$ acting on the fibers of $T^*(X)$ by dilation with character
$\hbar^{-1}$.
The motivation for the notation is that this map $ c_*^{\vee, \mathrm{Gi}}$ agrees with a 
map defined by Ginzburg in \cite[\S A.3]{ginzburg:characteristic} in a different way 
(and just denoted there by the notation $c_*$). Indeed, Ginzburg's class is
characterized by the requirement that its value at $[T_Z^* X]$ agrees with the
(signed) Chern-Mather class of $Z$; this identity is given in 
\cite[Lemma~A3.2]{ginzburg:characteristic}. The class defined in~\eqref{eq:Ginim1}
likewise satisfies
\begin{equation}\label{eq:GinSab}
 c_*^{\vee, \mathrm{Gi}} ([T_Z^* X]) = \cMa(Z)^\vee
\end{equation}
for all subvarieties $Z\subseteq X$, where $\cMa(Z)^\vee:=\chc_*(\cEu_Z)$ 
is the class obtained by changing the sign of the components
 of $\cMa(Z)$ of odd codimension  in $Z$,
cf.~\eqref{E:defdual}. Identity~\eqref{eq:GinSab}
follows from~\Cref{thm:zeropb}, by forgetting the $T$-action.
Therefore Ginzburg's natural transformation and ours agree in the non-equivariant setting,
and the composition $ c_*^{\vee, \mathrm{Gi}} \circ \CC = \chc_*$ coincides 
with the signed version of MacPherson's  natural transformation from constructible 
functions to homology (see also the work of Sabbah \cite{sabbah:quelques}).
 \Cref{thm:zeropb} generalizes this observation to
the equivariant setting.
}


\section{Preliminaries on cohomology of flag manifolds}\label{s:prelfm}
In what follows, we set up notation and recall basic facts about the
flag manifolds and their cohomology.  We will use the notation from
\cite{aluffi.mihalcea:eqcsm} and we refer the reader to
\cite{MR2143072} for further details.

\subsection{Schubert cells and varieties}
Let $G/B$ be the complete flag manifold, where $G$ is a complex
{semisimple} Lie group and $B$ is a Borel subgroup. Let $B^-$ be the
opposite Borel group and $T: = B \cap B^-$ the maximal torus.  The
Weyl group $N_G(T)/T$ is denoted by $W$, $\ell: W \to \Nbb$ is the
length function, and $w_0$ denotes the longest element. Notice that
$B^- = w_0 B w_0$. There is a root system $R$ associated to $(G,T)$
with simple roots $\Delta:=\{ \alpha_i \}_{1 \le i \le r }$ such that
$\alpha_i$ is positive with respect to $B$. The Weyl group $W$ is
generated by the simple reflections $s_i:=s_{\alpha_i}$.  A root
$\alpha \in R$ is positive if it can be written as a nonnegative
combination of simple roots; this will be denoted by $\alpha>0$.

For $w \in W$, define the Schubert cell
$X(w)^\circ:= Bw B/B \cong \C^{\ell(w)}$ and the opposite Schubert
cell $Y(w)^\circ:= B^- w B/B \cong \C^{\dim X - \ell(w)}$.  Their
closures give the Schubert variety $X(w) =\overline{B w B/B}$ and the
opposite Schubert variety $Y(w) = \overline{B^- w B/B}$. These are
complex projective algebraic varieties such that
$\dim_\C X(w) = \codim_\C Y(w) = \ell(w)$. The Bruhat order $\le$ is a
partial order on the Weyl group $W$; it may be defined by declaring
that $u \le v$ if and only if $X(u) \subseteq X(v)$.

More generally, let $P\subseteq G$ be a parabolic subgroup containing
$B$ and let $G/P$ be the corresponding partial flag manifold.
Let $W_P$ be the subgroup of $W$ generated by the simple reflections
in $P$ and denote by $W^P$ the {set of} minimal length representatives
for the cosets of {$W_P$ in $W$}. For each $w\in W$, $\ell(wW_P)$
denotes the length of the minimal length representative for the coset
$wW_P$. Let $w_P\in W_P$ be the longest element.  For each $w \in W^P$
there are Schubert cells $X(wW_P)^\circ := B w P/P$ and
$Y(wW_P)^\circ := B^- w P/P$ {in $G/P$, whose closures are} the
Schubert varieties $X(wW_P)$ and $Y(wW_P)$. Let $f: G/B \to {G/P}$ be
the natural projection. If $w \in W^P$ then $f$ restricts to
isomorphisms $X(w)^\circ\to X(wW_P)^\circ$, and
$f^{-1}(Y(w W_P)) = Y(w w_P)$. It follows that
$\dim_{\C} X(wW_P) = \codim_\C Y(wW_P) = \ell(w)$.  The Bruhat order
on $W$ restricts to a partial ordering on $W/W_P$ such that for
$u, v \in W$, $uW_P \le vW_P$ iff $X(uW_P) \subseteq X(vW_P)$. In
particular,
\begin{equation}\label{E:strat}
X(w W_P) = \bigsqcup_{wW_P \ge v W_P} X(v W_P)^\circ \quad \textrm{ and } \quad
Y(w W_P) = \bigsqcup_{wW_P \le vW_P} Y(v W_P)^\circ \/,
\end{equation}
thus the Schubert cells form a stratification of the corresponding
Schubert varieties.

\subsection{Schubert classes}
Since the varieties $G/P$ are smooth and projective, throughout this paper 
we will identify the equivariant homology and cohomology of $G/P$. In
particular, any $T$-stable, subvariety $Y \subseteq G/P$ of
complex codimension $c$ determines a fundamental class
$[Y]_T \in H^{2c}_T(G/P)$. We will omit the subscript $T$ for
non-equivariant classes.  By \cite[Proposition~2.1]{graham:positivity},
the identities in~\eqref{E:strat} imply that the (equivariant) fundamental classes
$\{ [X(wW_P)]_T \}_{w \in W^P}$ and $\{ [Y(wW_P)]_T \}_{w \in W^P}$
form $H^*_T(\pt)$-bases for the equivariant cohomology $H^*_T(G/P)$,
i.e.,
\[ H^*_T(G/P) = \bigoplus_{w \in W^P} H^*_T(\pt)  [X(wW_P)]_T = 
\bigoplus_{w \in W^P} H^*_T(\pt)  [Y(wW_P)]_T \/. \]
The opposite Schubert classes $[Y(wW_P)]_T$ are Poincar\'e
dual to the Schubert classes $[X(wW_P)]_T$, in the sense that
\begin{equation}\label{eq:pdS}
  \langle [X(uW_P)]_T,[Y(vW_P)]_T\rangle:=\int_{G/P} [X(uW_P)]_T\cdot [Y(vW_P)]_T 
  = \delta_{uW_P, vW_P}
\end{equation}
with respect to the usual intersection pairing (see e.g.,
\cite[Proposition~1.3.6]{MR2143072}). This holds since opposite
Schubert cells intersect generically transversally
\cite[Corollary~1.5]{richardson:double}.

Occasionally we will need to switch between $B$ and $B^-$ Schubert
data. This is done by utilizing the left multiplication by $w_0$. We
briefly recall the salient facts, and we refer the reader~e.g.,~to
\cite{knutson:noncomplex,MNS:left} for further details. Let
$n_{w_0} \in G$ be a representative of $w_0 \in W =
N_G(T)/T$. Left-multiplication by $n_{w_0}$ induces an automorphism
$\varphi_{w_0}: G/P \to G/P$, $gP \mapsto n_{w_0} g P$. This is {\em
  not} $T$-equivariant, but it is equivariant with respect to the
group automorphism $\chi_0:T \to T$ defined by
$\chi_0(t) = n_{w_0} t n_{w_0}^{-1}$. This means that
$\varphi_{w_0}(t.gP) = \chi_0(t).\varphi_{w_0}(gP)$. From
functoriality of equivariant cohomology, it follows that
$\varphi_{w_0}$ induces an automorphism
$\varphi_{w_0}^*: H^*_T(G/P) \to H^*_T(G/P)$, which `twists' the
coefficients in the base ring $H^*_T(\pt)$ according to the
automorphism $\chi_0$. Non-equivariantly, $\varphi_{w_0}^*$ is the
identity map.  Observe that since $H^*_G(G/P) = H^*_T(G/P)^W$,
$\varphi_{w_0}^*$ is a $H^*_G(G/P)$-algebra homomorphism, i.e., for any
$a \in H^*_G(G/P), b \in H^*_T(G/P)$,
$\varphi_{w_0}^*(a\cdot b) = a \cdot \varphi_{w_0}^*(b)$.  Our main
examples for classes in $H^*_G(G/P)$ will be the Chern classes of
homogeneous vector bundles on $G/P$.  Since $w_0^2 = \id$, it follows
that $\varphi_{w_0}^*$ is an involution.  For $\T = T \times \C^*$, with
$\C^*$ acting trivially,
the automorphism $\varphi^*_{w_0}$ can be
extended to one of $H^*_{\T}(G/P)$ by letting
$\varphi^*_{w_0}(\hbar ) = \hbar$.

Since $\varphi_{w_0}^{-1}(Y(w)^\circ) = X(w_0 w)^\circ$, it follows
that 
\begin{equation}\label{E:w0act} 
\varphi_{w_0}^*[Y(w)]_T = [X(w_0
  w)]_T \/; \quad \varphi_{w_0}^*(\csmT(Y(w)^\circ))
  = \csmT(X(w_0 w)^\circ) \/. 
\end{equation} 
  Finally, we observe that
$\varphi_{w_0}$ commutes with the $G$-equivariant projection
$f:G/B \to G/P$, therefore $\varphi_{w_0}^*$ commutes with the
pull-back $f^*$ and the push-forward $f_*$.


\section{Demazure-Lusztig operators}\label{sec:DLops}
In this section we recall the definition of the geometric version of
the Demazure-Lusztig (DL) operators which appear in the degenerate
Hecke algebra (cf.~ \cite{ginzburg:methods}). We also recall how
these operators determine the equivariant CSM classes
$\csmT(X(w)^\circ)$ and their duals (cf.~\cite{aluffi.mihalcea:eqcsm}).

\subsection{Definition and basic properties}\label{sec:basic-flags}
Recall the datum of $G \supset B \supset T$.  Let also $P_i \supset B$
be the (standard) minimal parabolic subgroup associated to the simple
root $\alpha_i$ and $\pi_i: G/B \to G/P_i$ the projection.

For each simple reflection $s_i \in W$ one can associate two operators
on $H^*_T(G/B)$. The first is the Bernste\u\i n-Gelfand-Gelfand
(BGG) operator
$\partial_i: H^{*}_T(G/B) \to H^{*- 2}_T(G/B)$, defined by
$\partial_i = \pi_i^* (\pi_i)_*$; cf.~ \cite{BGG}.  Since $\pi_i$ is
$G$-equivariant, $\partial_i$ is $H^*_T(\pt)$-linear. It satisfies
\[
  \partial_i [X(w)]_T = \begin{cases} [X(ws_i)]_T & ws_i >w \/; \\
    0 & \textrm{otherwise} \end{cases} \/;\,
  \partial_i [Y(w)]_T = \begin{cases} [Y(ws_i)]_T & ws_i < w \/; \\
    0 & \textrm{otherwise} \/. \end{cases}
\]
See e.g., \cite{MR2143072}. (The second equality follows from the first
by applying $\varphi_{w_0}^*$.)

The second operator is the algebra automorphism
$\mfs_i: H^*_T(G/B) \to H^*_T(G/B)$ obtained by the {\em
  right\/} Weyl group multiplication by (a representative of)
$s_i \in W$ on $G/T$. The projection $G/T \to G/B$ is a
$B/T$-bundle, and since $B/T$ (which is the unipotent group of $B$) is
equivariantly contractible, $G/T$ and $G/B$ have the same cohomology
groups.  From this definition it follows that $\mfs_i$ is
homogeneous and $H^*_T(\pt)$-linear.  One may show that
\[ 
\mfs_i = \id + c_1^T(\caL_{\alpha_i}) \partial_i \/, 
\] 
where $\caL_\alpha := G \times^B \C_\alpha$
is the homogeneous line bundle over $G/B$ with fiber over $1.B$ 
the $T$-module $\C_\alpha$ of weight $\alpha$. We refer 
to~\cite[\S2]{aluffi.mihalcea:eqcsm} (where a different sign
convention is used for~$\caL_{\alpha}$) for more details 
about $\mfs_i$.

For each simple reflection $s_i \in W$ define two non-homogeneous
operators:
\begin{equation}\label{E:DLops} 
\cT_i:= \partial_i - \mfs_i \/; \quad \cT_i^\vee:= \partial_i + \mfs_i \/. 
\end{equation} 
These are $H^*_T(\pt)$-linear operators acting on $H^*_T(G/B)$. The
`dual' operator~$\cT_i^\vee$ is precisely the Demazure-Lusztig
operator discussed by Ginzburg in \cite[(47)]{ginzburg:methods} in
relation with the degenerate Hecke algebra; see also
\cite{LLT:twisted,LLT:flagYB,lusztig:eqK}.  The operators
$\cT_i, \cT_i^\vee$ satisfy the braid relations for $W$ and
$\cT_i^2 = \id$ (\cite[Proposition~4.1]{aluffi.mihalcea:eqcsm}).  Thus,
we may define operators $\cT_w, \cT_w^\vee$ for any element $w$ of the
Weyl group. From this it follows that
\begin{equation}\label{E:twisted}
  \cT_u \cT_v = \cT_{uv} \/; \quad \cT_u^\vee \cT_v^\vee =
  \cT_{uv}^\vee \quad \forall v,w \in W \/,
\end{equation}
therefore these operators give a `twisted representation' of $W$ on
$H^*_T(G/B)$; cf.~\cite{LLT:twisted}.

Using the formulae for the action of $\partial_i, \mfs_i$
on Schubert classes one can
also write formulae for the action of $\cT_i,\cT_i^\vee$
(see~\cite[\S6.3]{aluffi.mihalcea:eqcsm}). We recall 
these, as they will be used in the proof of the orthogonality 
properties. The action of~$\cT_k$ is given by
{\small
\begin{equation}\label{E:LkX}
\cT_k ([X(w)]_{T})= 
\begin{cases} 
- [X(w)]_{T} & \textrm{if } \ell(ws_k) < \ell(w) \\  
& \\
\begin{aligned} 
&(1+ w(\alpha_k))[X(ws_k)]_{T} + [X(w)]_{T} \\ 
&+ \sum \langle \alpha_k, 
\beta^\vee \rangle [X(ws_ks_\beta)]_{T} 
\end{aligned} 
& \textrm{if } \ell(ws_k) > \ell(w)
\end{cases} 
\end{equation}}where the sum is over all positive roots $\beta \ne \alpha_k$ such that
$\ell(w) = \ell(ws_k s_\beta)$, and $w(\alpha_k)$ denotes the natural
$W$ action $w\cdot \alpha_k$ on $T$-weights; 
{note that $w(\alpha_k)>0$ since $ws_k > w$.}
{The action of the dual operator $\cT_k^\vee$ is given by}
{\small
\begin{equation}\label{E:LkdX}
\cT_k^\vee ([X(w)]_{T})= 
\begin{cases} 
[X(w)]_{T} & \textrm{if } \ell(ws_k) < \ell(w) \\  
& \\
\begin{aligned} 
& (1- w(\alpha_k))[X(ws_k)]_{T} - [X(w)]_{T} \\ 
& - \sum \langle \alpha_k, \beta^\vee \rangle [X(ws_ks_\beta)]_{T} 
\end{aligned} 
& \textrm{if } \ell(ws_k) > \ell(w)
\end{cases} 
\end{equation}}where the sum is as before. One may obtain similar formulae for the
actions on the opposite classes $[Y(w)]_T$ by using the
automorphism $\varphi_{w_0}^*$ from \Cref{E:w0act}.

The relevance of the DL operators comes from the following result,
proved in \cite[Theorem~6.4]{aluffi.mihalcea:eqcsm}.
\begin{theorem}\label{thm:ameqcsm} 
Let $w \in W$ be an element of the Weyl group. Then 
\[ 
\cT_i (\csmT (X(w)^\circ)) = \csmT(X(ws_i)^\circ) \/. 
\] 
Therefore, for every $w\in W$, we have
\[
\csmT(X(w)^\circ)=\cT_{w^{-1}}([X(\id)]_T)\/.
\]
\end{theorem}

Recall also (cf.~\eqref{E:defdual}) the definition of the {\em dual\/} CSM
  classes $\csmTv(X(w)^\circ)$, obtained from $\csmT(X(w)^\circ)$ by
changing signs of each homogeneous component according to
{\em co\/}dimension. Then \Cref{thm:ameqcsm} together with the definition of
the dual DL operators $\cT_i^\vee$ implies that
\begin{equation}\label{E:dualCSM}
  \csmTv(X(w)^\circ)=\cT_{w^{-1}}^\vee([X(\id)]_T)\/.
\end{equation}

\begin{remark}
  Similar statements hold for homogenized classes.  For instance, the
  homogenization of the operator $\cT_i$ is
  $\cT_i^\hbar:=\hbar \partial_i - \mfs_i$.  Recursive
  application of these operators yield the homogenization of the class
  $\csmT(X(w)^\circ)$. The homogenization of the dual class
  $\csmTv(X(w)^\circ)$ is obtained by applying
  $\cT_i^{\hbar,\vee}:=\hbar \partial_i+\mfs_i$. We leave the
  details to the reader.
\qede\end{remark}

\subsection{Adjointness}
Next we prove the key property in the proof of the `Hecke
orthogonality' in~\S\ref{sec:hecke-orth}. 
Recall that for $a, b \in H^*_T(G/B)$, $\langle a,b\rangle$ denotes $\int_{G/B} a\cdot b$.

\begin{prop}\label{lemma:DLadjoint} 
The operators $\cT_i$ and $\cT_i^\vee$ are adjoint to each other,
i.e., for any $a, b \in H^*_{T}(G/B)$ there is an identity in
$H^*_{T}(\pt)$:
\[ 
\langle \cT_i(a), b \rangle = \langle a, \cT_i^\vee(b) \rangle \/. 
\] 
Therefore, $\langle \cT_w(a), b \rangle = \langle a, \cT_{w^{-1}}^\vee(b) \rangle$
for all $w\in W$.
\end{prop}

\begin{proof} 
  It suffices to show that the BGG operator $\partial_i$ is
  self-adjoint and that the adjoint of $\mfs_i$ is
  $- \mfs_i$. We first verify that $\partial_i$ is
  self-adjoint{; while this is well-known, we include a proof for
    completeness.} Let $P_i$ be the minimal parabolic group and
  $\pi_i: G/B \to G/P_i$ the natural projection. Recall that
  $\partial_i = \pi_i^* (\pi_i)_*$. Then by the projection formula
\[
\langle \partial_i (a), b \rangle = \int_{G/B} \pi_i^* (\pi_i)_*(a)
\cdot b = \int_{G/P_i} (\pi_i)_* (a) \cdot (\pi_i)_*(b) = \langle a,
\partial_i (b) \rangle \/, 
\] 
where the last equality follows by symmetry.

In order to verify that $\mfs_i$ and $-\mfs_i$ are
  adjoint, let $e_w:= wB \in G/B$ denote the $T$-fixed point in $G/B$
corresponding to $w$ (so $e_{\id} =1.B$ is the $B$-fixed point). Then
$\int_{G/B} a \cdot b$ is the coefficient of
$[X(\id)]_T=[e_{\id}]_{T}$ in the expression for $a \cdot b$ with
  respect to the Schubert basis.  Recall also that
\[ 
\mfs_i [e_{\id}]_{T} = - [e_{s_i}]_{T} = P(t) [X(s_i)]_{T} -
[e_{\id}]_{T} 
\] 
where $P(t) \in H^2_{T}(\pt)$.  {(Cf.~e.g., \cite[(4) and 
\S6.3]{aluffi.mihalcea:eqcsm}.)}  Then
\[ 
\begin{split} 
  \langle \mfs_i(a), b \rangle & = \int_X \mfs_i (a) \cdot b 
  = \int_X \mfs_i(a) \cdot \mfs_i \mfs_i(b) = \int_{X} \mfs_i( a \cdot \mfs_i(b)) \\ 
  & = - \int_X a \cdot \mfs_i(b) = - \langle a, \mfs_i(b) \rangle \/,
\end{split}
\] 
using the fact that $\mfs_i$ is an $H^*_T(\pt)$-algebra homomorphism and squares to the
identity.
\end{proof}


\section{Orthogonality properties of CSM classes of Schubert
  cells}\label{s:homogenized}
In this section we prove two orthogonality results for equivariant CSM 
classes of Schubert cells in flag manifolds.  

The first, `geometric
orthogonality', states that the CSM classes of Schubert cells are
orthogonal to the Segre-MacPherson (SM) classes of opposite Schubert
cells. It will follow from the equivariant versions of general
transversality results from \cite{schurmann:transversality}, particularly
\Cref{thm:intersection}.

The second is the `Hecke orthogonality' mentioned in the
introduction. This holds only for complete flag manifolds $G/B$ and it
states that the CSM class are orthogonal to the dual/signed CSM
classes.  It is a consequence of the adjointness property from
\Cref{lemma:DLadjoint} and the fact that the CSM classes of Schubert
cells may be calculated by using DL operators.

As a consequence of these orthogonalites we obtain an identity among
the SM classes and signed CSM classes of Schubert cells in $G/B$;
cf.~\Cref{thm:dual}. This is a key ingredient used in \S\ref{s:pos} 
in the proof of a positivity property of CSM classes
conjectured in \cite{aluffi.mihalcea:eqcsm}. In addition, in \S
\ref{sec:conseq2} we use Hecke orthogonality to prove two
interesting properties about the Schubert expansion of CSM classes.

\subsection{The geometric orthogonality}\label{ss:geomorth}
Recall (cf.~\eqref{E:segre-new}) that the {Segre-Mac\-Pherson} (SM) 
class of a constructible function $\varphi \in \cFTinv(G/P)$ is defined 
by 
\[
\ssmT(\varphi):= \frac{\csT(\varphi)}{c^T(T(G/P))} \/.
\]

\begin{theorem}[Geometric orthogonality]\label{thm:gporth} 
Let $u, v \in W^P$. Then 
\[
  \left\langle \csmT(X(uW_P)^\circ), \ssmT(Y(vW_P)^\circ)
  \right\rangle = \delta_{u,v} \/.
\]
\end{theorem}

\begin{proof}
  The stratification by the $B$-orbits, and that by the $B^-$-orbits,
  are Whitney stratifications of $G/P$.  Indeed, the Whitney
  conditions hold generically on the $B$, respectively $B^-$ strata,
  and then by equivariance also everywhere along each orbit. Further,
  by \cite[Corollary~1.5]{richardson:double}, the intersections of of $B$
  and $B^-$-orbits are transversal to each other. Therefore we may
  apply \Cref{thm:intersection} to calculate
\[
\begin{split}
      \left\langle \csmT(X(uW_P)^\circ), \ssmT(Y(vW_P)^\circ) \right
      \rangle & = \int_{G/P} \csmT(X(uW_P)^\circ) \cdot
      \ssmT(Y(vW_P)^\circ) \\ & = \int_{G/P} \csmT(X(uW_P)^\circ \cap
      Y(vW_P)^\circ) \\ & = \chi(X(uW_P)^\circ \cap Y(vW_P)^\circ) \\
      & =\delta_{u,v} \/.
\end{split}
\] 
The last equality follows because if $u=v$ then the intersection
$X(uW_P)^\circ \cap Y(vW_P)^\circ$ is the single ($T$-fixed) point
$e_{uW_P}$, and if $u \neq v$ then the intersection is either empty or
a $T$-stable variety with no $T$-fixed points (see e.g., \cite[\S1.3]{MR2143072}), 
therefore its Euler
characteristic is equal to $0$.
\end{proof}

\subsection{The Hecke orthogonality}\label{sec:hecke-orth}
The goal of this subsection is to prove the following theorem:

\begin{theorem}[Hecke orthogonality]\label{thm:orthogonality}
  The equivariant CSM classes  of Schubert cells\/
  in $H^*_T(G/B)$ satisfy the following
  orthogonality {property:}
\[ 
\langle \csmT(X(u)^\circ), \csmTv(Y(v)^\circ) \rangle 
= \delta_{u,v} \prod_{\alpha >0} (1+ \alpha) \/. 
\]
{An analogous orthogonality property holds for opposite Schubert
  cells:}
\[ 
\langle \csmT(Y(u)^\circ), \csmTv(X(v)^\circ) \rangle 
= \delta_{u,v} \prod_{\alpha > 0} (1 - \alpha) \/. 
\]
\end{theorem}

In order to prove \Cref{thm:orthogonality} we need the following
lemma, which is a consequence of Theorem~\ref{thm:ameqcsm} and 
the formulae in~\eqref{E:LkX} and~\eqref{E:LkdX}; the proof is left to 
the reader.

\begin{lemma}\label{lem:lead}

For $w\in W$, let $e(w):=\prod_{\alpha>0,w^{-1}(\alpha)<0} (1+\alpha)$
and $\che(w):=\prod_{\alpha>0,w^{-1}(\alpha)<0} (1-\alpha)$. Then
\begin{align*}
  \csmT(X(w)^\circ) &= e(w) [X(w)]_T + \text{terms involving $[X(v)]_T$ for $v<w$;} \\
  \csmTv(X(w)^\circ) &= \che(w) [X(w)]_T + \text{terms involving $[X(v)]_T$ for $v<w$.}
\end{align*}
\end{lemma}

\begin{proof}[Proof of~\Cref{thm:orthogonality}] 
  To prove the first equality, observe that
  $\csmTv(Y(v)^\circ) = \cT_{v^{-1} w_0}^\vee [Y(w_0)]_T$; this
  follows from identity~\eqref{E:dualCSM} by applying the automorphism
  $\varphi_{w_0}$. By~\Cref{lemma:DLadjoint} and identities~\eqref{E:twisted} 
  and \eqref{E:dualCSM}, we have
\begin{align*}  
  \langle \csmT(X(u)^\circ), \csmTv(Y(v)^\circ) \rangle 
  &= \langle \cT_{u^{-1}}[X(\id)]_T, \cT_{v^{-1} w_0}^\vee [Y(w_0)]_T
    \rangle\\
  &=\langle \cT_{w_0 v} \cT_{u^{-1}} [X(\id)]_T, [Y(w_0)]_T \rangle \\
  &= \langle\cT_{w_0 v u^{-1}} [X(\id)]_T , [Y(w_0)]_T \rangle \\
  &= \text{coefficient of $[X(w_0)]_T$ in $\csmT(X(uv^{-1}w_0)^\circ)$} \/.
\end{align*} 
By~\Cref{lem:lead}, this coefficient is $0$ unless $u=v$, and it
equals $\prod_{\alpha>0} (1+\alpha)$ if $u=v$. This verifies the first
equality.  The second equality follows from the first, by applying
the automorphism $\varphi_{w_0}^*$.
\end{proof}

\begin{corol}[CSM Poincar{\'e} duality]\label{cor:CSMPd}
  Ordinary CSM classes are Poincar\'e dual to dual CSM classes of
  opposite cells. That is:
\begin{equation}\label{eq:pdC}
\langle \csm(X(u)^\circ), \csmve(Y(v)^\circ) \rangle = \delta_{u,v} \/. 
\end{equation}
\end{corol}

\begin{proof} 
  This follows from the previous theorem by specializing
  $\alpha \mapsto 0$.
\end{proof}
In ordinary homology, the leading terms of $\csm(X(u)^\circ)$ and
$\csmve(Y(v)^\circ)$ are $[X(u)]$, $[Y(v)]$, respectively: we may view
these CSM classes as `deformations' of the fundamental classes by
lower dimensional terms.  ~\Cref{cor:CSMPd} states that these
deformations preserve the intersection pairing: cf.~\eqref{eq:pdS}
and~\eqref{eq:pdC}.

\subsection{Consequences of orthogonality I: equality of SM and dual
  CSM classes}
Combining the geometric and Hecke orthogonalities from
\Cref{thm:gporth} and \Cref{thm:orthogonality}, together with the fact
that the Poincar{\'e} pairing is non-degenerate, we obtain the main
result of this section:

\begin{theorem}\label{thm:dual} 
Let $v\in W$. Then the following equality holds in $H_T^*(G/B)$:
\[
\csmTv(X(v)^\circ)
=\bigl(\prod_{\alpha \in R_+} (1 - \alpha)\bigr) \ssmT(X(v)^\circ).
\]
In particular, this yields the following identity in $H^*(G/B)$: 
\begin{equation}\label{cor:cTXcve}
  \csmve(X(v)^\circ) = \ssm(X(v)^\circ) \/.
\end{equation}
\end{theorem}

This identity is one of the two key ingredients in the proof
of the positivity of CSM classes of Schubert cells. The other is provided
by the theory of $\caD$-modules, which we will use in~\S\ref{s:CSMCC}
to prove that the signed SM class is effective.

\begin{example}\label{ex:Fl23}
  For $X=\Fl(2)=\Pbb^1$ and $v=w_0$ the longest Weyl group element,
  $\csmve(X(v)^\circ) = [\Pbb^1] - [\pt]$,
  $\csm(X({v})^\circ) = [\Pbb^1] + [\pt]$ and
  $c(T\Pbb^1) = [\Pbb^1] + 2[\pt]= 1 + 2[\pt]$.
  Then
  \[
    \ssm(X(v)^\circ) = \frac{[\Pbb^1] +[\pt]}{1+2[\pt]} =
    (1-2[\pt])([\Pbb^1] +[\pt]) = [\Pbb^1] - [\pt] \/,
  \]
  verifying identity~\eqref{cor:cTXcve} in this case.

  More generally, an algorithm calculating CSM classes of Schubert
  cells $X(v)^\circ$ (and therefore their duals as well) was obtained
  in~\cite{aluffi.mihalcea:eqcsm}, and this may be used to verify {the
    identities from \Cref{thm:dual}} explicitly in many concrete
  cases.  For instance, the following are the (non-equivariant) CSM
  classes of the Schubert cells in $\Fl(3)$, the variety parametrizing
  flags in $\C^3$:
\[ 
\begin{split} 
\csm(X(w_0)^\circ) & = [\Fl(3)] + [X(s_1 s_2)] + [X(s_2 s_1)]
+ 2 [X(s_1)] + 2 [X(s_2)] + [\pt]; \\ 
\csm(X(s_1 s_2)^\circ) & = [X(s_1 s_2)] + [X(s_1)] + 2 [X(s_2)] + [\pt]; \\ 
\csm(X(s_2 s_1)^\circ)& = [X(s_2 s_1)] +2 [X(s_1)] + [X(s_2)] + [\pt]; \\ 
\csm(X(s_1)^\circ) & = [X(s_1)] + [\pt]; \\ 
\csm(X(s_2)^\circ) & = [X(s_2)] + [\pt]; \\ 
\csm(X(\id)^\circ) &= [\pt] \/. 
\end{split} 
\] 
The total Chern class of $\Fl(3)$ is 
\[ 
\begin{split}
  c(T \Fl(3))  =& \sum_v \csm(X(v)^\circ)= [\Fl(3)] +2 [X(s_1 s_2)] +
  2[X(s_2 s_1)] + 6 [X(s_1)] \\ &+ 6 [X(s_2)] +6 [\pt]
\end{split}
\] 
Again one can check identity~\eqref{cor:cTXcve}, by using the 
multiplication table in $H^*(\Fl(3))$.
\qede\end{example}

\begin{example}\label{ex:GPsegre}
  {View $\Pbb^2$ as a partial flag manifold. The Schubert cells are
    isomorphic to $\Abb^i$, $i=0,1,2$, and we have
\[
\csm(\Abb^2) = \csm(\Pbb^2)-\csm(\Pbb^1) =[\Pbb^2]+2[\Pbb^1]+[\Pbb^0]\quad.
\]
The cell $\Abb^2$ is {\em not\/} its own opposite, yet
\[ 
\begin{split}
\langle \csm(\Abb^2),\csmve(\Abb^2)\rangle_{\Pbb^2}
& =\int ([\Pbb^2]+2[\Pbb^1]+[\Pbb^0])\cdot ([\Pbb^2]-2[\Pbb^1]+[\Pbb^0]) \\ 
& =1-4+1 = -2\ne 0\quad. 
\end{split}
\]
Further, note that identity~\eqref{cor:cTXcve} does not extend to the parabolic
case. Indeed,
\[
\frac{\csm(\Abb^2)}{c(T\Pbb^2)} = [\Pbb^2]-[\Pbb^1]+[\Pbb^0] \ne \csmve(\Abb^2)\/.
\]
Finally, note that
\[
\begin{split}
\left\langle \csm(\Abb^2),\frac{\csm(\Abb^2)}{c(T\Pbb^2)}\right\rangle_{\Pbb^2}
& =\int ([\Pbb^2]+2[\Pbb^1]+[\Pbb^0])\cdot ([\Pbb^2]-[\Pbb^1]+[\Pbb^0]) \\ 
& =1-2+1 = 0 
\end{split}
\]
as we would expect from the geometric orthogonality in \Cref{thm:gporth}.} 
\qede\end{example}

\subsection{Consequences of orthogonality II: the transition matrix
  between Schubert and CSM classes}
\label{sec:conseq2} 
We present next two consequences of the Hecke orthogonality property,
in the non-equivariant setting.

The CSM class of each Schubert cell may be written in terms of the
Schubert basis:
\begin{equation}\label{eq:cuvs}
\csm(X(v)^\circ) = \sum_{u\in W} \cc(u;v) [X(u)]\quad.
\end{equation}
with $\cc(u;v)\in \Z$. A natural question is to find the inverse of
the matrix $( \cc(u;v) )_{u,v \in W}$.
\begin{prop}\label{prop:cuvinverse}
The inverse of the matrix $\big( \cc(u;v) \big)_{u,v}$ is the matrix
\[ \big( (-1)^{\ell(u)-\ell(v)} \cc(w_0 v; w_0 u) \big)_{u,v} \quad \/. \]
\end{prop}

\begin{proof} Let $(d(u;v))_{u,v}$ be the inverse matrix. In other
  words,
  \[
    [X(v)] = \sum_{u \in W} d(u;v) \csm(X(u)^\circ) \quad \/.
  \] 
  By \Cref{cor:CSMPd},
  $d(u;v) = \langle [X(v)], \csmve (Y(u)^\circ) \rangle$.  This is the
  coefficient of $[Y(v)]$ in the expansion of $\csmve (Y(u)^\circ)$ in
  the basis of (opposite) Schubert classes. The statement follows from
  the definition of dual CSM classes and the fact that $[Y(w)] = [X(w_0 w)]$
  for every $w \in W$.
\end{proof}

\begin{example}\label{ex:Fl4}
For $\Fl(3)$ we consider the matrix $A$ whose $(i,j)$ entry is the coefficient
  $\cc(w_i,w_j)$, where we list the~permutations $w_i\in S_3$,
  $i=1,\dots, 6$ in the order 
  \[ \id, s_1,s_2,s_1 s_2, s_2 s_1, s_1 s_2 s_1 \/. \]
  From \Cref{ex:Fl23} (see also  
  the `non-equivariant' part of the matrix shown in
\cite[Example~6.7]{aluffi.mihalcea:eqcsm}), the matrix $A$ and its inverse 
are given by:
\[
A=\begin{pmatrix}
1 & 1 & 1 & 1 & 1 & 1 \\
0 & 1 & 0 & 1 & 2 & 2 \\
0 & 0 & 1 & 2 & 1 & 2 \\
0 & 0 & 0 & 1 & 0 & 1 \\
0 & 0 & 0 & 0 & 1 & 1 \\
0 & 0 & 0 & 0 & 0 & 1 
\end{pmatrix} \/,\quad 
A^{-1} = \begin{pmatrix}
1 & -1 & -1 & 2 & 2 & -1 \\
0 & 1 & 0 & -1 & -2 & 1 \\
0 & 0 & 1 & -2 & -1 & 1 \\
0 & 0 & 0 & 1 & 0 & -1 \\
0 & 0 & 0 & 0 & 1 & -1 \\
0 & 0 & 0 & 0 & 0 & 1 
\end{pmatrix}\/.
\]
The second matrix is (up to signs) the anti-transpose
of the first one. This is the content of~\Cref{prop:cuvinverse}
in this example.

Larger examples may be also computed explicitly, making use 
of~\cite[Corollary~4.2]{aluffi.mihalcea:eqcsm}. A previous version
of this article, available on the~ar$\chi$iv, presented the two
$24\times 24$ matrices for the case of $\Fl(4)$.
\qede\end{example}

Consider now the problem of defining a constructible
function~$\Theta_V$ on a given variety $V$, such that
$c_*(\Theta_V) = [V]$. Such a function is of course not uniquely
defined, but there are interesting situations in which a particularly
natural function satisfies this property. For example, if $V$ is a
toric variety compactifying a torus $V^\circ$, then
$\Theta_V=\one_{V^\circ}$ is such a function
(\cite[Th\'eor\`eme~4.2]{MR2209219}).  \Cref{prop:cuvinverse} implies
that Schubert varieties offer another class of examples.

\begin{corol}\label{cla:55}
Let $\Theta_v=\sum_u (-1)^{\ell(v)-\ell(u)} \cc(w_0v;w_0u) \one_{X(u)^\circ}$.

Then $c_*(\Theta_v)=[X(v)]$.
\end{corol}

\begin{proof}
  This follows immediately from~\eqref{eq:cuvs} and \Cref{prop:cuvinverse}.
\end{proof}

To illustrate, let $w=w_0$, the maximum length element in the Weyl
group.  Then according to Corollary~\ref{cla:55}
\[
[G/B] = c_*\left(\sum_{v} (-1)^{\ell(w_0)-\ell(v)}\one_{X(v)^\circ}\right) \quad.
\]
This identity is independently proven
in~\cite[Proposition~5.5]{aluffi.mihalcea:eqcsm}.

\begin{remark}
  Consider the Schubert expansion of the equivariant CSM class:
  \[
    \csmT(X(w)^\circ) = \sum \cc^T(v;w) [X(v)]_T \/,
  \]
where
  $\cc^T(v;w) \in H^*_T(\pt) = Sym_{\Z} \mathfrak{X}(T)$.
  There is an interpretation of the coefficients
  $\cc^T(u;v)$ in the affine nil-Hecke algebra, found by S.J. Lee in
  \cite{sj:chern}. We briefly recall this interpretation, and refer to
  {\em loc.~cit.} for all the details. The affine nil-Hecke algebra
  $\caH_{\textit{nil}}$ is generated by the elements
  $\bar{\partial}_w$ and {$\lambda \in \mathfrak{X}(T)$}, 
subject to the following relations:
\begin{enumerate}
\item
  $\lambda\mu=\mu\lambda$, for any $\lambda,\mu \in
  {\mathfrak{X}(T)}$;
\item
  $\bar{\partial}_w\bar{\partial}_y=\delta_{\ell(wy),\ell(w)+\ell(y)}\bar{\partial}_{wy}$;
\item
  $\bar{\partial}_i \lambda=s_i\lambda\cdot\bar{\partial}_i-\langle
  \lambda,\alpha_i^\vee\rangle$.
\end{enumerate}
For each simple root $\alpha_i$, define the element
$\bar{s}_i:= 1 + \alpha_i \bar{\partial}_i \in
\caH_{\textit{nil}}$. The elements $\bar{\partial}_i$
and $\bar{s}_i$ satisfy the braid relations in the Weyl group
$W$. Let $w= s_{i_1} \cdot \ldots \cdot s_{i_k} \in W$ be a reduced
decomposition. Then according to \cite[Theorem~6.2]{sj:chern} there is an
identity
\[
  (\bar{s}_{i_1} + \bar{\partial}_{i_1}) \cdot \ldots \cdot
  (\bar{s}_{i_k} + \bar{\partial}_{i_k}) = \sum_{v} \cc^T(v;w)
  \bar{\partial}_v \/.
\]
Geometrically, $\bar{\partial}_i$ corresponds to the BGG operator
$\partial_i$, $\bar{s}_i$ to $-\mfs_i$, and the weight
$\lambda$ to the Chevalley multiplication by the equivariant Chern
class $c_1^T(G \times^B \C_{-\lambda})$ (the class of a
$G$-equivariant line bundle over $G/B$). In the related work
\cite{su2017k}, Su, Zhao and Zhong observe a relation between the
affine Hecke algebra and a K-theoretic version of stable envelopes/CSM
classes.
(Also see~\S\ref{ss:csmvsstable}.)
\qede\end{remark}


\section{CSM classes and characteristic cycles for flag
  manifolds}\label{s:CSMCC}

In the classical results of Beilinson-Bernstein
\cite{beilinson.bernstein:localisation}, Brylinski-Kashiwara
\cite{brylinski.kashiwara:KL}, and Kashiwara-Tanisaki
\cite{kashiwara.tanisaki:characteristic}, the theory of characteristic
cycles associated to holonomic $\caD$-modules on flag manifolds
becomes a powerful geometric tool to study the representation theory
related to the Kazhdan-Lusztig polynomials.
  Characteristic cycles of $\caD$-modules and those of
  constructible functions are closely related. There are multiple sign
  conventions in the literature, and in this section we carefully
  spell out the relation and conventions used in this paper;
  essentially, we require that characteristic cycles of holonomic
  $\caD$-modules are effective. Then we utilize
  \Cref{thm:zeropb} to find the precise relation between CSM classes
  of Schubert cells and the characteristic cycles of Verma
  modules. The main result is~\Cref{thm:csmsegre}, which is also the
  main ingredient in the proof of the positivity of CSM classes. As an
  application, we define certain `Kazhdan-Lusztig' classes in 
  \S\ref{sec:KLclasses}, and show they are also positive.

\subsection{CSM classes and characteristic cycles of Verma
  modules}\label{CSMVerma}
Let $X=G/B$ be the generalized flag manifold.
We recall some results from \cite{kashiwara.tanisaki:characteristic},
and in order~to satisfy the hypotheses from {\em loc.cit.}~we 
assume in addition that $G$ is simply connected.\begin{footnote} 
{Note that for complex semisimple $G$, the flag
    varieties $G/P$, and the Schubert cells, only depend on the Lie
    algebra of $G$ (see e.g.,~\cite[\S3.1]{chriss.ginzburg:representation}), 
    so this assumption is harmless for our (co)homological 
    calculations.}\end{footnote}
Let $\rho \in \mathfrak{X}(T)$
denote half the sum of positive roots. For $w \in W$ let $M_w$ be the
Verma module of highest weight $-w\rho - \rho$, a module over the
universal enveloping algebra $U(\g)$; see \cite[p.~291]{HTT}. Let
$\cM_w$ denote the holonomic $\caD_X$-module
\[ 
\cM_w = \caD_X \otimes_{U(\g)} M_w \/. 
\] 
Consider the constructible complex $\DR(\cM_w)$. According to
\cite[Theorem~3]{kashiwara.tanisaki:characteristic} (where it is
attributed to Brylinski-Kashiwara \cite{brylinski.kashiwara:KL} and
Beilinson-Bernstein~\cite{beilinson.bernstein:localisation}) there is
an identity
\[
\DR(\cM_w) = \C_{X(w)^\circ} [\ell(w)]\,;
\]
where the right-hand side is the shifted constant sheaf on $X(w)^\circ$,
also cf.~\cite[Corollary~12.3.3(i)]{HTT}.  (Note that the definition
of $\DR$ from \cite{kashiwara.tanisaki:characteristic} differs from
the one from \cite{ginzburg:characteristic} and~\cite{HTT} by a shift
of $\dim X$.)  It follows that the constructible function associated
to the Verma module $M_w$ is
\[ 
\chi_\text{stalk}(\DR(\cM_w)) = (-1)^{\ell(w)} \one_{X(w)^\circ}\/. 
\] 
By the commutativity of diagram~\eqref{E:diagram},
\begin{equation}\label{E:Char-CC}
\Char(\cM_w) = (-1)^{\ell(w)} \CC(\one_{X(w)^\circ})\/;
\end{equation}
therefore, \Cref{cor:CSMpb} implies the following result. 
Let $\csmTh(X(w)^\circ):=\csmT(X(w)^\circ)^\hbar$ denote the 
$\hbar$-homogenization of degree $\dim G/B$.
\begin{corol}\label{cor:csmcc} 
Let $w \in W$. Then
\[ 
\csmTh(X(w)^\circ) = (-1)^{\ell(w)} \iota^* [\Char(\cM_w)]_{T\times \C^*}
\]
where $\iota: G/B \to T^*(G/B)$ is the zero-section and $\C^*$ acts on 
the fibers of~$T^*(G/B)$ by the character $\hbar^{-1}$.
\end{corol} 

\subsection{The CSM class as a Segre class: proof of the main
  theorem}\label{s:pos}
In this section we prove Theorem \ref{thm:main} from the
introduction. Recall that this is the main ingredient to prove that in
the non-equivariant case the CSM classes are effective, thus proving
the positivity conjecture stated in \cite{aluffi.mihalcea:eqcsm}.  We
start by proving the following lemma, which is known among experts,
but for which we could not find a reference.

\begin{lemma}\label{lemma:chernprod}
  The following equality holds in $H^*_T(G/B)$:
  \[
    c^T(T(G/B)) \cdot c^T(T^*(G/B)) = \prod_{\alpha >0} (1 -
    \alpha^2) \/.
  \]
\end{lemma}

\begin{proof}
  The Chern class of $G/B$ is given by
  $c^T(T(G/B)) = \prod_{\alpha > 0} (1 +
  c_1^T(\caL_{-\alpha}))$ where (as in~\S \ref{sec:basic-flags})
  $\caL_{-\alpha} = G \times^B \C_{-\alpha}$.
Since the localization of $\caL_{-\alpha}$ at the $T$-fixed point $e_w$ is 
$w(-\alpha)$, it follows that $c^T(T(G/B))_{|w} = \prod_{\alpha >0} (1 - w(\alpha))$.
 From this we obtain that
  \[ (c^T(T(G/B)) \cdot c^T(T^*(G/B)))_{|w} = \prod_{\alpha >0} (1 -
    w(\alpha)) (1 + w(\alpha)) = \prod_{\alpha >0} (1 -
    \alpha)(1+\alpha) \/, \] because $w$ permutes the set of roots.
\end{proof}

Recall the following construction from \S\ref{sec:eqshadows}. Consider the
projection $\overline{q}: \Pbb( T^*(G/B) \oplus \one) \to G/B$ and the
tautological subbundle
$\cO_{T^*(G/B) \oplus \one} (-1) \subseteq T^*(G/B) \oplus \one$; this
is an inclusion of $T$-equivariant bundles. If $C$ is a 
 $T$-stable
cycle in
$T^*(G/B)$, then the Segre operator acts on $C$ by
\[
  \Segre^T(C) = \overline{q}_* \left(
  \frac{[\overline{C}]}{c^T(\cO_{T^*(G/B) \oplus \one} (-1))}\right)
\] 
where $\overline{C}$ denotes the closure of $C$
 in $\Pbb(T^*(G/B) \oplus \one)$; cf.~\eqref{E:defsegre}. 
This operator takes values in the completion $\hat{H}^T_*(G/B)$. It
follows from the next result that 
 in the cases of interest here it in fact takes values in
the localization
at the non-zero elements in $H^*_T(\pt)$.

\begin{theorem}\label{thm:csmsegre}
  Let $w \in W$. The following equality holds:
  \[
    \csmT(X(w)^\circ) = \Bigl(\prod_{\alpha > 0} (1 + \alpha)\Bigr)\,
    \Segre^T (\Char( \cM_w )) \/.
  \]
\end{theorem}

\begin{proof}
We have
\begin{align*}
\csmT(X(w)^\circ) &\overset{\eqref{E:segre-new}}= c^T(T(G/B))\cap \ssmT(X(w)^\circ) \\
&\overset{\mathrm{Thm.}\ref{thm:dual}}= \frac{c^T(T(G/B))}
{\prod_{\alpha > 0}(1 - \alpha)}\cap \csmTv(X(w)^\circ) \\
&\overset{\eqref{E:defdual}}= (-1)^{\dim X(w)^\circ}\, \frac{c^T(T(G/B))}
{\prod_{\alpha > 0}(1 - \alpha)}\cap \chcT(\one_{X(w)^\circ}) \\
&\overset{\mathrm{Cor.}\ref{cor:segree}}= (-1)^{\ell(w)}\, \frac{c^T(T(G/B)) \cdot c^T(T^*(G/B))}
{ \prod_{\alpha > 0} (1-\alpha)}\cap \Segre^T (\CC(\one_{X(w)^\circ})) \\
&\overset{\mathrm{Lem.}\ref{lemma:chernprod}}=  (-1)^{\ell(w)}
\Bigl(\prod_{\alpha > 0} (1 + \alpha)\Bigr) \Segre^T (\CC(\one_{X(w)^\circ})) \\
&\overset{\eqref{E:Char-CC}}= \Bigl(\prod_{\alpha > 0} (1 + \alpha)\Bigr) 
\Segre^T (\Char( \cM_w ))
\end{align*}
as stated.
\end{proof}

As explained in introduction, \Cref{thm:csmsegre} implies that the
non-equivariant CSM classes of Schubert cells are effective. For the
convenience of the reader we recall the statement, adding also a
positivity statement of Segre-MacPherson classes which follows
immediately.

\begin{corol}[Positivity of CSM classes]\label{cor:pos}
  (a) Let $X=G/P$ be a generalized flag manifold and $w \in W$. Then
  the non-equivariant CSM class $\csm(X(wW_P)^\circ)$ is effective,
  i.e., in the Schubert expansion
  \[
    \csm (X(wW_P)^\circ) = \sum_{vW_P \le wW_P} c(vW_P;wW_P) [X(vW_P)]
    {~\in H_*(X)}\/,
  \] 
the coefficients $c(vW_P;wW_P)$ are non-negative. 

(b) Let $X=G/B$ and $w \in W$. Then the Segre-MacPherson class
$\ssm(X(w)^\circ)$ is Schubert alternating, i.e., in the Schubert
expansion
\[
  \ssm (X(w)^\circ) = \sum_{v \le w} d(v;w) [X(v)] {~\in H_*(X)}\/,
\] 
the coefficients $d(v;w)$ satisfy $(-1)^{\ell(w)-\ell(v)} d(v;w) \ge 0$.
\end{corol}

\begin{proof}
  Part (a) follows from~\Cref{thm:csmsegre}, as explained in the
  introduction. Part (b) follows from (a) and identity~\eqref{cor:cTXcve}.
\end{proof}

\begin{remark}
  In \cite{AMSS:ssmpos} we utilize the methods in this paper to
  extend the statement of part (b) to any flag manifold $G/P$. If
  $G/P$ is a Grassmannian, this was conjectured in \cite{feher.rimanyi:chern-schwartz-macpherson},
  see \S1.5 and Conjecture~8.4.
\qede\end{remark}

\subsection{Kazhdan-Lusztig classes}\label{sec:KLclasses}
In analogy to \Cref{cor:csmcc}, in this section we define {\em
  Kazhdan-Lusztig (KL) classes} associated to the intersection
cohomology (IC) complex. We show that the KL classes of Schubert
varieties are positive. In some important situations, the KL classes
equal to the Mather classes.

Let $X$ be a smooth complex algebraic variety, and $Y \subseteq X$ a
subvariety. Denote by $\IC(Y)  \in\Perv(X)$ the intersection cohomology
(IC) complex of the subvariety $Y$, {with the convention that the
  restriction to the regular part $Y^\text{reg}$ is the shifted constant
  sheaf $\mathbb{C}_{Y^\text{reg}}[\dim Y]$.}

\begin{defin}\label{def:KL}
  The {\em Kazhdan-Lusztig (KL) class} of $Y$, denoted by
  $\KL(Y)$, is defined by
  \[
    \KL(Y):= (-1)^{\dim Y} \csT(\chi_\text{stalk}(\IC(Y)))  \in
    H_*^T(X)\/.
  \]
\end{defin}

Now let $X=G/B$ and $Y= X(w) \subseteq G/B$ a Schubert variety.  The
Riemann-Hilbert correspondence gives an equality
$\DR(\caL_w) = \IC(X(w))$, where $\caL_w := \caD_X \otimes_{U(\g)} L_w$ is the 
holonomic $\caD_X$-module associated to $L_w$, the quotient of the Verma
module $M_w$ by its maximal proper submodule; see
e.g.,~\cite[Theorem~3]{kashiwara.tanisaki:characteristic} (or
\cite[(12.2.13) and Corollary~12.3.3(ii)]{HTT}).  
By the commutativity of diagram~\eqref{E:diagram}, 
$\Char(\caL_w) = \CC(\chi_{stalk}(\IC(X(w))))$;
then by \Cref{thm:zeropb} it follows that 
\[ \KL(X(w))^\hbar = (-1)^{\ell(w)} \iota^*[\Char(\caL_w)] \/. \]
By the proof of the
Kazhdan-Lusztig conjectures
\cite{beilinson.bernstein:localisation,brylinski.kashiwara:KL} (see
also \cite[Chapter~12]{HTT}) we have
\[ 
\Char(\caL_w) = \sum_{u \le w} (-1)^{\ell(w)- \ell({u})}
P_{u,w}(1) \Char(\cM_{u}) 
\] 
where $P_{u,w}(q)$ is the Kazhdan-Lusztig polynomial. 
By~\Cref{cor:csmcc}, this proves:

\begin{prop}\label{prop:KLschub}
  The KL class of $X(w)$ is given by:
\[
\KL(X(w))  = \sum_{u \le w} P_{u,w}(1)\,
\csmT(X(u)^\circ)  \quad \in H_*^T(G/B)\/. 
\]
\end{prop}

Equivalently,
\begin{equation}\label{E:KLconsfun}
(-1)^{\ell(w)} \chi_\text{stalk}(\IC(X(w))) 
= \sum_{u\le w} P_{u,w}(1) \one_{X(u)^\circ}\/.
\end{equation}
\begin{corol}\label{cor:KLpos}
  The non-equivariant KL class is strongly Schubert effective. That
  is, the coefficients $k(u;w)$ from the Schubert expansion
  \[
    \KL(X(w)) = \sum_{u \le w} k(u;w) [X(u)] \/,
  \] 
  satisfy $k(u;w) >0$.\end{corol}

\begin{proof}
  This follows from the positivity of CSM classes
  (\Cref{cor:intropos}), observing that $[X(u)]$ is the initial term
  of the CSM class $\csm(X(u)^\circ)$, combined with the fact that the
  KL polynomials have non-negative coefficients and $P_{u,w}(1) \ge 1$
  (see e.g., \cite[Theorem~13.2.11]{HTT} and \cite[\S7.11]{humphreys:reflection}).
\end{proof}

\begin{remark} 
  \Cref{prop:KLschub} and \Cref{cor:KLpos} hold more generally for
  Schubert varieties in partial flag manifolds $G/P$. Indeed, the
  projection $f:G/B \to G/P$ induces a pull-back
  $f^*:\cF(G/P) \to \cF(G/B)$ defined by
  $f^*(\varphi)(g.B) = \varphi(g.P)$. Since $f$ is a smooth morphism,
\[
\begin{split} 
f^*\chi_\text{stalk}(\IC(X(wW_P))) & =
        \chi_\text{stalk}(\IC(f^{-1}(X(wW_P))[\dim P/B]) \\ 
& =(-1)^{\dim P/B} \chi_\text{stalk}(\IC(X(w w_P))) \/, 
\end{split}
\]
  where $w_P$ is the longest element in $W_P$.  One combines this with
  the fact that the parabolic Kazhdan-Lusztig polynomials coincide
  with those from $G/B$; see Deodhar's results \cite[Proposition~3.4
  and Theorem~4.1]{deodhar:bruhatII}. Details are left to the reader.
\qede\end{remark}

For general flag manifolds $G/P$, there are situations in which
the characteristic cycle of the IC sheaf $\IC(X(wW_P))$ is known to be
irreducible, and hence equal to the conormal cycle
$T^*_{X(wW_P)}(G/P)$.  For example, this is the case for the IC sheaf
of Schubert varieties in the ordinary Grassmannians, by a result of
Bressler, Finkelberg, and Lunts \cite{MR1084458}; see~\cite{MR1451256}
for generalizations. In this case, by \Cref{cor:CSMpb},
\[
\KL(X(wW_P))^\hbar = 
(-1)^{\ell(wW_P)}\iota^*[\CC(\chi_\text{stalk}(\IC(X(wW_P)))] 
= \cMaT(X(wW_P))^{\hbar}\/.
\]

In such cases, the following interesting equality holds:
\begin{equation}\label{E:P=KL}
P_{uW_P,wW_P}(1) = \Eu_{X(wW_P)}(p)
\end{equation}
for $p\in X(uW_P)^\circ$, $uW_P\le wW_P$.~
Here~(recall)~$\Eu_{X(wW_P)}$ is MacPherson's local Euler obstruction.
Indeed, if the characteristic cycle of the IC sheaf is irreducible,
then it must equal the conormal cycle, so this follows 
from~\eqref{eq:EoCc} and the extension of~\eqref{E:KLconsfun}
to $G/P$.
The equality in~\eqref{E:P=KL} may also be deduced from the
microlocal index formula (\cite[Theorem~6.3.1]{MR725502},
\cite[Th\'eor\`eme~3]{MR86e:32016a}); see also \cite[Remark~5.0.4 on
pp.~294-295]{schurmann:book} and \cite[Theorem~3.9]{schurmann:lecture}.
We refer to \cite{jones:csm,MR1451256,mihalcea.singh:conormal} for
calculations of the local Euler obstruction for Schubert varieties in
(cominuscule) Grassmannians.

\section{CSM classes and stable envelopes}\label{ss:csmvsstable}
Stable envelopes were introduced by Maulik and Okounkov
\cite{maulik.okounkov:quantum} in their study of symplectic
resolutions, and in relation to integrable systems; see also the
series of papers by Rim{\'a}nyi, Tarasov and Varchenko
\cite{RTV:Yangian,RTV:Kstable,RTV:partial}.

Maulik and Okounkov \cite{maulik.okounkov:quantum} remark that 
in the case of $T^*(G/B)$, the stable
envelopes are given by classes of certain conic Lagrangian cycles. 
We give an outline of the proof of this fact, by
identifying them to characteristic cycles for Verma modules (up to
sign); cf.~\Cref{lemma:stab}. With this, we deduce that the pullback
of the stable envelopes to the zero section $G/B$ coincide with the
CSM classes of the Schubert cells (\Cref{cor:csmstab}).  This allows
us to create a fruitful dictionary between the stable envelopes theory
and that of CSM classes, which we use to deduce a localization
formula for the CSM classes (\Cref{cor:localizations}) and a Chevalley
formula (\Cref{thm:che for csm}). We also generalize these results to
the case of partial flag manifolds.

\begin{remark}
The dictionary between stable envelopes and characteristic classes may 
also be used to give another proof of the geometric orthogonality from
\Cref{thm:dual}.
This proof was given in an initial version of this paper on the ar$\chi$iv.
\qede\end{remark}

\subsection{Definition of stable envelopes}
We recall the definition of stable 
envelopes for $T^*(G/B)$; see
\cite[Chapter~3]{maulik.okounkov:quantum} or \cite{su:restriction} for
more details.

The action of the torus $T$ on $G/B$ extends to one on $T^*(G/B)$.
There is an additional dilation action by $\C^*$ 
on the fibers of
$T^*(G/B)$ with 
a character~$\chi$, which we choose to be $\chi=\hbar^{-1}$, coinciding with
the conventions in \cite{maulik.okounkov:quantum,su:restriction}.
Explicitly, $\C^*$ acts on the fibers of $T^*(G/B)$ by
$z.(x,\xi):=(x,z^{-1}\xi)$, where
$z\in \C^*$, $x\in G/B$ and $\xi\in T_x^*(G/B)$. 
The $T\times \C^*$
fixed points in $T^*(G/B)$ are $\{(e_w,0) \}$ for $w\in W$. For any
$w\in W$ and $\gamma\in H_{T\times\C^*}^*(T^*(G/B))$, we denote by
$\gamma|_w$ the restriction of $\gamma$ to the fixed point $(e_w,0)$.
As seen from \Cref{definition of stable basis} below, 
the definition of stable envelopes depends on 
the choices of the character $\chi$ and of
a Weyl chamber in the Lie algebra of the maximal torus.
We let $+$ denote the positive chamber determined by the Borel subgroup $B$, and $-$
denote the opposite chamber.  
Our choice for $\chi$ also coincides with that from \Cref{cor:CSMpb} and~\Cref{cor:csmcc}
above, and it leads to natural identities
between localizations of CSM classes and stable envelopes; cf.~ \Cref{cor:csmstab}.

The 
$+$ version of
stable envelopes 
is characterized
by the following theorem.
\begin{theorem}[\cite{maulik.okounkov:quantum,
    su:restriction}]\label{definition of stable basis}
  There exist unique classes
  \[ 
  \{\stab_+(w)\in H_{T\times\C^*}^{2 \dim G/B}(T^*(G/B)) \,|\, w\in W\} 
  \] 
which satisfy the following properties:
\begin{enumerate}
\item $\stab_+(w)$ is supported on
  $\bigcup_{u\leq w} T_{X(u)}^*(G/B)$, i.e., $\stab_+(w)|_u = 0$ unless
  $u \leq w$;
\item
$\stab_+(w)|_w=\prod\limits_{\alpha>0,w\alpha<0}(w\alpha-\hbar)
\prod\limits_{\alpha>0,w\alpha>0}w\alpha$;
\item $\stab_+(w)|_u$ is divisible by $\hbar$, for any $u<w$ in the
  Bruhat order.
\end{enumerate}
\end{theorem}

\begin{remark}\label{rem:stable}
  (1) By the first two properties, the transition matrix between
  $\{\stab_+(w)\,|\, w \in W\}$ and the fixed point basis in the
  localized cohomology
  \[
    H_{T\times\C^*}^*(T^*(G/B))_\text{loc}:=H_{T\times\C^*}^*(T^*(G/B))
    \otimes_{H_{T\times\C^*}^*(\pt)} \Frac H_{T\times\C^*}^*(\pt)
  \]
  is triangular with nontrivial diagonal terms. Hence the stable
  envelopes $\{\stab_+(w)\,|\, w\in W\}$ form a basis in
  $H_{T\times\C^*}^*(T^*(G/B))_\text{loc}$, called the {\em stable basis\/}
  for $T^*(G/B)$.

(2) Similarly, there are stable envelopes
  $$\{\stab_-(w)\in H_{T\times\C^*}(T^*(G/B)) \,|\, w\in W\}$$ for the
  negative chamber, satisfying the following analogous properties:\\ 
(a)  $\stab_-(w)$ is supported on $\bigcup_{u\geq w} T_{Y(u)}^*(G/B)$;\\ 
(b) 
  $\stab_-(w)|_w=
  \prod\limits_{\alpha>0,w\alpha>0}(w\alpha-\hbar)\prod\limits_{\alpha>0,w\alpha<0}w\alpha$; and\\
 (c) $\stab_-(w)|_u$ is divisible by $\hbar$, for any $u>w$ in
  the Bruhat order.
\qede\end{remark}

The following was observed by Maulik and Okounkov \cite[p.~69,
Remark~3.5.3]{maulik.okounkov:quantum}, but for completeness we
include a sketch of the proof, using \Cref{cor:csmcc}.

\begin{lemma}\label{lemma:stab} 
For any $w \in W$, 
\[
[\Char(\cM_w)] = (-1)^{\dim X - \ell(w)} \stab_+(w)\in H_{T\times\C^*}^*(T^*(G/B))\/. 
\] 
\end{lemma}

\begin{proof}[Sketch of the proof] 
  We need to check that conditions (1)--(3) in Theorem \ref{definition
    of stable basis} are satisfied. The support condition (1) follows
  from the definition of the characteristic cycle. To check (2) and
  (3) we notice first that $\iota: X \to T^*(X)$ is
  $T \times \C^*$-equivariant, and that the fixed loci satisfy
  $(T^*(X))^{T \times \C^*} = X^T$, since $\C^*$ acts trivially on
  $X$. By~\Cref{cor:csmcc}, for every $u \le w$ the
  localization of $[\Char(\cM_w)]$ is given by
\[
[\Char(\cM_w)]|_{u} = (-1)^{\ell(w)} \csmTh(X(w)^\circ)|_{u}
\/. 
\] 
The homogenized CSM class can be written as
\[
\csmTh(X(w)^\circ) = \sum_{ u \le w, 0 \le k } \hbar^k (c^T(u;w) [X(u)])_k 
\]
where $c^T(u;w)$ are polynomials in $H^*_T(\pt)$ of degree
$\le \ell(u)$ (cf.~\cite[Proposition~6.5(a)]{aluffi.mihalcea:eqcsm})
and $(c^T(u;w) [X(u)])_k$ is the component of $c^T(u;w) [X(u)]$ in
$H_{2k}^T(X)$.  Since $wB \notin X(u)$ for $u < w$, the localization
$[X(u)]|_{w}$ is equal to $0$.  It follows that the localization
$\Char(\cM_w)|_w$ equals the homogenization
\[ 
\Char(\cM_w)|_w = (-1)^{\ell(w)} \csmTh(X(w)^\circ)|_{w} =
(-1)^{\ell(w)} (c^T(w,w) [X(w)]|_{w})^\hbar \/.
\] 
The coefficient $c(w;w)$ is calculated in \Cref{lem:lead}, and the
localization $[X(w)]|_{w}$ is the Euler class of the normal bundle of
$X(w)$ at the smooth point $w$; see e.g.,~\cite[\S2]{knutson:noncomplex} 
for a combinatorial formula for this.
We leave it as an exercise to use these formulae in order to check the
correct normalization from condition (2). (Similar results were also
obtained by Rim{\'a}nyi and Varchenko \cite{rimanyi.varchenko:csm}
using Weber's localization formulae \cite{Weber}.)

Let $c_0:= \csmT(X(w)^\circ)_{ \deg 0}
\in H_0^T(X)$
be the degree $0$ part of the
non-homogenized class $\csmT(X(w)^\circ)$.
To check condition (3) it
suffices to show $(c_0)|_u = 0$ for any $u < w$. By
\cite[Proposition~6.5(d)]{aluffi.mihalcea:eqcsm} $c_0 = [e_w]$, the
equivariant class of the $T$-fixed point $w$. Clearly ${[e_w]}|_u = 0$
for $u \neq w$ and this finishes the proof.
\end{proof}

\subsection{CSM classes, stable envelopes and localization formulae}\label{ss:Csealf}
From \Cref{cor:CSMpb} and \Cref{lemma:stab}, we obtain immediately the
following formula; a different proof may be found in
\cite{rimanyi.varchenko:csm}.

\begin{prop}\label{cor:csmstab} 
Let $w \in W$ be a Weyl group element. Then 
\[ 
\iota^*(\stab_+(w)) = (-1)^{\dim X} \csmTh(X(w)^\circ) \/.
\]
\end{prop}

We can also relate the stable basis associated to the negative chamber
$\{\stab_{-}(w)\mid w\in W\}$ and CSM classes associated to the
opposite Schubert cells.  This uses the extension of the
automorphism $\varphi_{w_0}^*$ from \Cref{E:w0act} to
$\varphi_{w_0}^*: H^*_{T \times \C^*}(T^*(G/B)) \to H^*_{T \times
  \C^*}(T^*(G/B))$.  The extended automorphism preserves $\hbar$
(since $\mathbb{C}^*$ acts trivially on $G/B$), and it acts on
$H^*_T(\pt)$ by $w_0$.  
The following follows immediately from the
definition of the stable envelopes, see \Cref{rem:stable}(2).

\begin{lemma}\label{lemma:w0} 
The automorphism $\varphi_{w_0}^*$ satisfies 
\[ 
\varphi_{w_0}^*(\stab_+(w)) = \stab_{-}(w_0 w) \/. 
\] 
\end{lemma}

We then obtain a parallel to Proposition~\ref{cor:csmstab}.  Denote by
$\csmThv(Y(w)^\circ)$ the $\hbar$-homogenization of the dual class
$\csmTv(Y(w)^\circ)$.

\begin{prop}\label{cor:i-} 
The following equalities hold: \begin{itemize}
\item[(i)] $
\iota^* ( \stab_{-}(w) )|_{\hbar \mapsto -\hbar} 
= (-1)^{\ell(w)} \csmThv(Y(w)^\circ)$;
\item[(ii)] $\iota^* ( \stab_{-}(w) ) = (-1)^{\dim X} \csmTh(Y(w)^\circ)$.\end{itemize}
\end{prop}

\begin{proof}
  This is a standard calculation, using \Cref{cor:csmstab} and
  \Cref{lemma:w0}.
\end{proof}

In \cite{su:restriction}, the last-named author found localization
formulae for the stable envelopes at any torus fixed point, in the
process generalizing formulae of Anderson-Jantzen-Soergel/Billey 
\cite{andersen1994representations,billey:kostant} for the
localization of Schubert classes. \Cref{cor:i-} implies similar
localization formulae for the homogenized CSM classes. We record this
next. Recall that $\alpha_i$ denote the simple roots for $(G,B,T)$.

\begin{corol}\label{cor:localizations}
  Fix $u,w \in W$ two elements such that $w \le u$ in Bruhat ordering,
  and fix a reduced decomposition
  $u = s_{i_1} \cdot \ldots \cdot s_{i_\ell}$.  Then the localization
  $\csmTh (Y(w)^\circ)|_{u}$ equals

\begin{equation}\label{formula restriction -B}
  \csmTh (Y(w)^\circ)_{|u}=(-1)^{\dim G/B - \ell(u)}\prod\limits_{\alpha\in R^+\setminus R(u)}(\alpha-\hbar)\sum
  \hbar^{\ell-k}\prod\limits_{t=1}^k \beta_{j_{t}},
\end{equation}
where the sum is over all subwords
$s_{i_{j_1}}s_{i_{j_2}}\dots s_{i_{j_k}}$ of
$u = s_{i_1} \ldots s_{i_\ell}$ such that
$w=s_{i_{j_1}}s_{i_{j_2}}\dots s_{i_{j_k}}$; for $1\leq t\leq l$,
$\beta_t:=s_{i_1}s_{i_2}\dots s_{i_{t-1}}\alpha_{i_t}$ with
$\beta_1 = \alpha_{i_1}$; and $R(u)=\{\beta_i |1\leq i\leq \ell\}$.
\end{corol}

Note that the set $R(u)$ coincides with the set of {\em inversions} of
$u^{-1}$, i.e., the set of those positive roots $\alpha$ such that
$u^{-1}(\alpha) <0$; cf. \cite[p.~14]{humphreys:reflection}. Moreover,
the sum in the equation \eqref{formula restriction -B} does not depend
on the reduced expression for $u$, see \cite{su:restriction}. A
similar formula for the localization of the CSM class
$\csmTh(X(w)^\circ)$ can be obtained by applying the automorphism
$\varphi_{w_0}^*$ to \eqref{formula restriction -B} and using
\Cref{E:w0act}.

\subsection{Partial flag {manifolds}}\label{s:partial}
In this section we generalize the above relation between CSM classes
and stable envelopes in the case of partial flag manifolds.  Our main
result is \Cref{prop:push}. We also use this and a Chevalley
formula for stable envelopes to deduce a Chevalley formula for CSM
classes (\Cref{thm:che for csm}).

There is an analogue of the Existence \Cref{definition of stable
  basis} which yields the set of stable envelopes
$\{\stab_{\pm}^P (u): u \in W^P\}$, with the $\pm$ sign denotes the
positive/negative Weyl chamber. For each sign choice, the
corresponding set forms a basis for the cohomology ring
$H^*_{T \times \C^*}(T^*(G/P))$ localized at
$H^*_{T \times \C^*}(\pt)$. See e.g.,~\cite{su:restriction} for more
details.  We can consider diagram \eqref{E:diagram} for $G/P$; there 
are well defined push-forwards for each of the (associated Grothendieck) 
groups in this diagram, and as in~\S\ref{ss:Ginz} the given
maps commute with (proper) push-forward.  The next proposition states
that the relation between stable envelopes and CSM classes
extends to partial flag manifolds.  
Let $f: G/B \to G/P$ be the natural map.

\begin{prop}\label{prop:push} 
Let $w \in W^P$ be a minimal length representative. The following
hold:
\begin{itemize}
\item[(a)]The constructible function associated to the direct image
  complex $f_*(\cM_w)$ of the holonomic module $\cM_w$ is
  $(-1)^{\ell(w)}\one_{X(wW_P)^\circ}$.
\item[(b)] There is an identity 
\[ 
\iota^* (\stab_+^P(w)) = (-1)^{\dim (G/P)}\csmTh (X(wW_P)^\circ) \/. 
\]
\end{itemize}
\end{prop}

\begin{proof} 
  Using that $\DR$ and $\chi_\text{stalk}$ commute with push-forward we
  obtain
\[ 
f_*( (\chi_\text{stalk} \circ \DR) \cM_w) = (-1)^{\ell(w)} f_*
(\one_{X(w)^\circ}) \/. 
\] 
Then (a) follows because for $w \in W^P$ the restriction
$f: X(w)^\circ \to X(w W_P)^\circ$ is an isomorphism. Part (b) follows
because 
\[ \stab_+^P(w)=(-1)^{\dim(G/P)-\ell(w)}[\Char(f_*(\cM_w))] \/, \]
(with a proof similar to that from \Cref{lemma:stab}), and from
~\Cref{cor:CSMpb}.
\end{proof}

We now turn to the Chevalley formula. This formula gives the Schubert
expansion of a product of a divisor Schubert class $[Y(s_\beta)]$ by
another class $[Y(w)]$, in an appropriate cohomology ring of $G/P$;
here $s_\beta \in W \setminus W_P$ is a simple reflection and
$w \in W^P$.  We refer to \cite{fulton.woodward}, in the
non-equivariant setting, and e.g., to \cite[\S8]{buch.m:nbhds} for the
formula in the equivariant ring $H^*_T(G/P)$.  For stable
envelopes, a Chevalley formula was found by the last-named author in
\cite[Theorem~3.7]{su:quantum}.  Then \Cref{prop:push} determines a
formula to multiply the CSM class $\csmTh(Y(w)^\circ)$ by any divisor
class. We record the result next. Let $\varpi_\beta$ be the
fundamental weight corresponding to the simple root $\beta$, and let
$R_P^+$ denote the set of positive roots in $P$.

\begin{theorem}\label{thm:che for csm}
  Let $w\in W^P$, and $\beta$ be a simple root not in $P$. Then the
  following identity holds in $H_{T\times \mathbb{C}^*}^*(G/P)$:
\[ 
\begin{split}
[Y(s_\beta)]_T \cup  \csmTh (Y(w)^\circ)
=&(\varpi_\beta-w(\varpi_\beta))\csmTh (Y(w)^\circ) \\ 
& +\hbar\sum (\varpi_\beta, \alpha^\vee) 
\csmTh (Y(ws_\alpha W_P)^\circ)\/, 
\end{split}
\]
where the sum is over roots $\alpha\in R^+\setminus R^+_P$ such that
$\ell(ws_\alpha W_P)> \ell(w)$, $\alpha^\vee$ is the coroot of
$\alpha$, and $( \cdot , \cdot )$ is the evaluation pairing.
\end{theorem}

The classical Chevalley formula Theorem can be deduced from \Cref{thm:che
  for csm} via a limiting process as follows. Write
\[
  \csmTh (Y(w)^\circ)=\sum_{u\geq w}c^\hbar(u;w)[Y(u)]_{T \times
    \mathbb{C}^*} \/,
\]
where $ u \in W^P$ and the coefficients
$c^\hbar(u;w) \in H_{T\times \mathbb{C}^*}^{2(\dim G/P-
  \ell(u))}(\pt)$.  using this and \Cref{lem:lead},
  we deduce that
\[
  \lim_{\hbar\rightarrow \infty}\frac{\csmTh
    (Y(w)^\circ)}{(\hbar)^{\dim G/P-\ell(w)}}=[Y(w)]_T \/.
\]
For any root $\alpha \in R^+\setminus R^+_P$, such that
$\ell(w s_\alpha W_P) > \ell(w)$ we have
\[
  \lim_{\hbar\rightarrow \infty}\frac{\hbar\csmTh (Y(ws_\alpha
    W_P)^\circ)}{(\hbar)^{\dim G/P- \ell(w)}}=[Y(ws_\alpha W_P)]_T
\]
if and only if $\ell(ws_\alpha W_P)=\ell(w)+1$. Otherwise, the limit
is 0. Hence, if we divide both sides of the equation in \Cref{thm:che
  for csm} by $(\hbar)^{\dim G/P- \ell(w)}$, and let $\hbar$ go to
$\infty$, we obtain
\[
  [Y(s_\beta)]_T \cup  [Y(w)]_T
=(\varpi_\beta-w(\varpi_\beta))[Y(w)]_T
+\sum (\varpi_\beta, \alpha^\vee) 
[Y(ws_\alpha W_P)]_T\/,
\]
where the sum is over those roots $\alpha \in R^+ \setminus R_P^+$
such that $\ell(ws_\alpha W_P) = \ell(w)+1$. This is the classical
Chevalley formula; see e.g., \cite[Theorem~8.1]{buch.m:nbhds}.


\section{Appendix: Non-characteristic pullback results}\label{s:app}

In this Appendix we explain how \Cref{cor:segree} allows us to extend
a `non-characteristic pullback formula' of~\cite[Theorem~1.4, (3.12), 
(3.16)]{schurmann:transversality} for the (signed)
Segre-MacPherson classes to the $T$-equivariant context. Let
$f: X\to Y$ be a $T$-equivariant morphism of smooth complex algebraic
varieties with a $T$-action. In order to recall the definition of {\em
  non-characteristic}, we need to introduce the following commutative
diagram (whose right square is cartesian, 
see~\cite[Diagram~(2.10)]{schurmann:transversality}):
\begin{equation*}
\xymatrix@R=25pt@C=25pt{
T^*(X) \ar[d]^{\pi_X} & f^*(T^*(Y)) \ar[l]_t \ar[r]^{f'} \ar[d]^{\pi'} & T^*(Y) \ar[d]^{\pi_Y} \\
X \ar@{=}[r] & X \ar[r]^f & \:\:Y\:.
}
\end{equation*}
Here $f'$ is the map induced by base change, whereas $t$ is the dual
of the differential of $f$.  Then $f$ is by definition {\em
  non-characteristic} with respect to a closed conic subset
$C\subseteq T^*(Y)$ (i.e., a closed algebraic subset stable under the
$\C^*$-action given by multiplication on the fibers of the vector
bundle $T^*(Y)$) if
\begin{equation*}
f'^{-1}(C) \cap \ker(t) \subseteq \iota'(X)\:,
\end{equation*}
with $\iota': X\to f^{*}(T^*(Y))$ the zero-section of the vector bundle $f^{*}(T^*(Y))$.
By~\cite[Lemma~3.2]{schurmann:transversality} this is
\%marginpar{(84)}
equivalent to requiring that $t: f'^{-1}(C)\to T^*(X)$ is proper and therefore
finite. If $C$ is moreover $T$-stable, then $C$, $f'^{-1}(C)$ and
$C':=t( f'^{-1}(C))\subseteq T^*(X)$ are $\T$-stable for
$\T=T\times \C^*$ resp., $\T=T$ and one gets an induced group
homomorphism
\begin{equation}\label{pullback-cone}
t_*\circ f'^*: H_*^{\T}(C)\to H_*^{\T}(C')\:.
\end{equation}
Here we use the $\T$-equivariant map $f': f^*T^*(Y)\to T^*(Y)$ of ambient
smooth complex algebraic varieties for the refined Gysin map
\[ f'^*: H_*^{\T} (C)\to H_*^{\T}(f^{-1}(C)) \/.\]

We will use the group homomorphism \eqref{pullback-cone} for suitable 
characteristic cycles in the following theorem, which holds for both choices 
$\T=T\times \C^*$ and  $\T=T$. The case  $\T=T$ is used in Theorem~\ref{thm:non-cc2}
and Corollary~\ref{cor:intersection}, whereas the case $\T=T\times \C^*$ is used in 
Theorem~\ref{thm:intersec-CC}.

\begin{theorem}\label{thm:non-cc}
  Let $f: X\to Y$ be a $T$-equivariant morphism of smooth complex algebraic
  varieties of dimension $m=\dim X, n=\dim Y$.  Assume that $f$ is
  non-characteristic with respect to the support
  $C:=\supp(\CC(\gamma))\subseteq T^*(Y)$ of the characteristic cycle
  $\CC(\gamma)$ of a $T$-invariant constructible function
  $\gamma\in \cFTinv(Y)$. Then $C':=t(f'^{-1}(C))$ is pure
  $m$-dimensional, with
\begin{equation*}
t_*f'^*(\CC(\gamma)) = (-1)^{m-n}\cdot \CC(f^*(\gamma))\:.
\end{equation*}
In particular, the left hand side is a Lagrangian cycle in $T^*(X)$,
i.e., belongs to $L_T(X)$.
\end{theorem}

\begin{proof}
  Forgetting the $T$-action, this is
  \cite[Theorem~3.3]{schurmann:transversality}. If
  $\gamma\in \cFTinv(Y)$, then this is in fact an equality of
  $\T$-stable cycles for $\T=T\times \C^*$ resp., $\T=T$.
\end{proof}

Consider now the corresponding commutative diagram of $T$-equivariant
{\em projective completions}:
\[
\vcenter{\xymatrix@R=40pt@C=40pt{
\Pbb(T^{*}X\oplus\one) \ar[d]^{\overline{\pi}_X} &
U\subseteq \Pbb(f^{*}(T^{*}Y)\oplus\one) \ar[l]_-{\overline t}
\ar[r]^-{\overline f} \ar[d]^{\overline \pi'} &
\Pbb(T^{*}Y\oplus\one) \ar[d]^{\overline{\pi}_Y} \\ 
X \ar@{=}[r] & X \ar[r]^f & Y\/.
}}
\]
The right square is cartesian, but the map $\overline{t}$ is only
defined on the complement $U$ of $\Pbb(\ker(t)\oplus\{0\})$.  For the
application to Segre classes it is important to note that
\[
\overline{t}^{*}(\cO_{\Pbb(T^{*}X\oplus\one) }(-1))\cong 
(\overline{f}^{*}(\cO_{\Pbb(T^{*}Y\oplus\one) }(-1)))|U\:.
\]
 In the context of \Cref{thm:non-cc} one has
 $\overline{f}^{-1}(\overline{C})\subseteq U$ and
 $\overline{t}\left(\overline{f}^{-1}(\overline{C})\right)=\overline{C'}$.
 Then a simple calculation gives
\begin{eqnarray}\label{eq:non-cc}
  f^*\left(\Segre^T(\CC(\gamma))\right)
  &=& \Segre^T\left(t_*f'^*(\CC(\gamma)) \right)\\
 &=& (-1)^{m-n}\cdot \Segre^T\left( \CC(f^*(\gamma))\right)
  \in \hat{H}^T_*(X) \nonumber
 \end{eqnarray}
 exactly as in the non-equivariant context in
 \cite[(3.12)]{schurmann:transversality}. 
These classes correspond to {\em signed\/} equivariant
Segre-MacPherson (SM) classes (as in~\eqref{E:signed-segre}).
In terms of the `unsigned' classes \eqref{E:segre-new}, this proves
the following result.

\begin{theorem}\label{thm:non-cc2}
  Let $f: X\to Y$ be a $T$-equivariant morphism of smooth complex algebraic
  varieties.  Assume that $f$ is non-characteristic with respect to
  the support $\supp(\CC(\gamma))\subseteq T^*(Y)$ of the characteristic
  cycle $\CC(\gamma)$ of a $T$-invariant constructible function
  $\gamma\in \cFTinv(Y)$. Then
\begin{equation*}
f^*\left(\ssmT(\gamma)\right) = \ssmT(f^*(\gamma)) \in \hat{H}^T_*(X) \:.
\end{equation*}
\end{theorem}
The equivariant intersection formula used in this paper, 
\Cref{thm:intersection}, is an equivariant version 
of~\cite[Theorem~1.2]{schurmann:transversality}. It may be
proved by applying \Cref{thm:non-cc2} to the diagonal inclusion
$d: X\to X\times X$ of a smooth complex algebraic variety $X$ with a
$T$-action. Then $d$ is $T$-equivariant for the diagonal $T$-action on
$X\times X$.  Let $\alpha, \beta\in \cFTinv(X)$. Then
$\alpha\cdot \beta= d^*(\alpha\boxtimes \beta)$ for
$\alpha\boxtimes \beta \in \cFTinv(X\times X)$ defined by
$\alpha\boxtimes \beta(x,x'):=\alpha(x)\cdot \beta(x')$. Moreover
\begin{equation*}
\CC(\alpha\boxtimes \beta)=\CC(\alpha)\boxtimes \CC(\beta)\:,
\end{equation*}
since $\Eu_{Z\times Z'}=\Eu_Z\boxtimes \Eu_{Z'}$ \cite[Eq.~3 on
p.~426]{macpherson:chern}.  Similarly,
\begin{equation}\label{eq:multiplicative-c}
\csT(\alpha\boxtimes \beta)=\csT(\alpha)\boxtimes \csT(\beta)\:,
\end{equation}
e.g., by the corresponding multiplicativity of the $T$-equivariant
Chern-Mather classes defined via the $T$-equivariant Nash blow-up
\cite[\S4.3]{ohmoto:eqcsm}. This yields:

\begin{corol}\label{cor:intersection}
  Assume the $T$-equivariant diagonal embedding $d: X\to X\times X$ of
  a smooth complex algebraic variety $X$ is non-characteristic with respect
  to the support $\supp(\CC(\alpha\boxtimes \beta))$ of the
  characteristic cycle $\CC(\alpha\boxtimes \beta)$ for two
  $T$-invariant constructible functions
  $\alpha,\beta\in \cFTinv(X)$. Then
\begin{equation*}
\csT(\alpha\cdot \beta)= \csT(\alpha) \cdot  \ssmT(\beta) \in H^T_*(X)\subseteq \hat{H}^T_*(X)\:.
\end{equation*}
 If $X$ is moreover compact, then
\begin{align*}
  \langle  \csT(\alpha), \ssmT(\beta) \rangle 
  &:=\int_X  \csT(\alpha) \cdot  \ssmT(\beta) \\
  &= \int_X \csT(\alpha\cdot \beta) =  \int_X c^T_0(\alpha\cdot \beta) =\chi(X;\alpha\cdot \beta) \/.
\end{align*}
\end{corol}

\begin{remark}
  Let $\alpha\in \cFTinv(X)$, resp., $\beta\in \cFTinv(X)$ be
  constructible with respect to algebraic Whitney stratifications
  $\cS:= \{ S \subseteq X \}$, resp.,
  $\cS':= \{ S' \subseteq X \}$ of $X$,
 i.e., $\alpha|S$ and $\beta|S'$ are constant for all strata $S\in \cS$ and 
$S'\in \cS'$. Assume
 that all strata
  $S\in \cS$ are {\em transversal} to all strata
  $S'\in \cS'$,  i.e., for all $x\in S\cap S'$ one has the equality 
$T_x(S)+T_x(S')=T_x(X)$.
This is also equivalent to requiring that 
the diagonal
$d(X)$ is transversal in $X\times X$ to all product strata $S\boxtimes S'$.
  Then $d$ is non-characteristic with respect to the support
  $\supp(\CC(\alpha\boxtimes \beta))$ of the characteristic cycle
  $\CC(\alpha\boxtimes \beta)$ (see~\cite{schurmann:transversality}).
\qede\end{remark}

As a final application we get the following intersection formula
for characteristic cycles fitting with the corresponding
  orthogonality relation for stable envelopes 
(and see~\cite[Corollary~3.5]{schurmann:transversality} for the 
corresponding non-equivariant result).

\begin{theorem}\label{thm:intersec-CC}
  Let $X$ be a compact smooh complex algebraic variety with a $T$-action
  and $\C^*$ acting on the fibers of $T^*(X)$ by the character
  $\hbar^{-1}$.  Assume that the diagonal embedding
  $d: X\to X\times X$ is non-characteristic with respect to
  $\supp(\CC(\alpha\times \beta))$ for some given
  $\alpha, \beta\in \cFTinv(X)$.  Then $\CC(\alpha)\cdot \CC(\beta)$
  is supported on the zero-section $\iota(X)\subseteq T^*(X)$, with
\begin{eqnarray*} 
\langle [\CC(\alpha)]_{T \times \C^*}, [\CC(\beta)]_{T \times \C^*} \rangle
&:=&  \int_{T^*(X)} [\CC(\alpha)]_{T \times \C^*}
\cdot [\CC(\beta)]_{T \times \C^*}\\
&=& (-1)^{\dim X}\cdot\chi(X;\alpha\cdot \beta) \:. \nonumber
\end{eqnarray*} 
\end{theorem}

\begin{proof}
Consider the cartesian diagram
\begin{equation*}
\xymatrix@R=30pt@C=30pt{
T^*(X) \ar[r]^-{(\id,a)} \ar[d]_\pi & T^*(X)\times_X T^*(X) \ar[d]^t \ar[r]^-{d'} & T^*(X\times X) \\
X \ar[r]_\iota & T^*(X) \:,
}
\end{equation*}
with $\iota: X\to T^*(X)$ the zero-section, $a: T^*(X)\to T^*(X)$ the
antipodal map and $f: X\to \pt$ the proper constant map.  By the
non-characteristic assumption,
\begin{equation*}
t: d'^{-1}(\supp(\CC(\alpha)\times \CC(\beta)))\to T^*(X)
\end{equation*}
is proper. Then, by base change, the restriction of $\pi$
\begin{equation*}
\supp(\CC(\alpha))\cap \supp(a_*\CC(\beta)) \subseteq 
(\id,a)^{-1}d'^{-1}(\supp(\CC(\alpha)\times \CC(\beta)))\overset\pi\to X
\end{equation*}
is proper. Since $a_*\CC(\beta)=\CC(\beta)$, the $\C^*$-stable
subset
\begin{equation*}
\supp(\CC(\alpha)) \cap \supp(\CC(\beta))
\end{equation*}
has to be contained in the zero-section $\iota(X)\subseteq T^*(X)$. By
\Cref{thm:non-cc} and base-change one gets:
\begin{align*}
\langle [\CC(\alpha)]_{T \times \C^*}, [\CC(\beta)]_{T \times \C^*} \rangle 
& =f_*\pi_*(\id,a)^*d'^*\left([\CC(\alpha)]_{T \times \C^*}\boxtimes 
[\CC(\beta)]_{T \times \C^*} \right) \\
&= f_*\pi_*(\id,a)^*d'^*[\CC(\alpha\boxtimes \beta)]_{T \times \C^*} \\
&=f_*\iota^*t_*d'^*[\CC(\alpha\boxtimes \beta)]_{T \times \C^*}\\
&= (-1)^{\dim X}\cdot f_*\iota^*[\CC(\alpha\cdot \beta)]_{T \times \C^*} \:.
\end{align*}
Since
\begin{equation*} 
\iota^*[\CC(\alpha\cdot \beta)]_{T \times \C^*} = \csT(\alpha\boxtimes\beta)^\hbar
\end{equation*}
by \Cref{thm:zeropb}, and
\[
  f_*\iota^*[\CC(\alpha\cdot \beta)]_{T \times \C^*}= \int_X
  \csT(\alpha\boxtimes\beta)^\hbar = \chi(X;\alpha\cdot \beta) \/.
\]
by functoriality.
\end{proof}


\bibliographystyle{halpha}
\bibliography{AMSSbiblio}

\end{document}